\documentclass[12pt]{amsart}
\usepackage[utf8]{inputenc}
\usepackage{amsmath,amssymb,amsthm, graphics}
\usepackage{amsthm}
\usepackage{diagmac2}
\usepackage{textcomp}
\usepackage[alphabetic]{amsrefs}
\usepackage[all,cmtip]{xy}
\usepackage{tikz-cd}
\usepackage{yfonts}
\usepackage{fullpage}
\usepackage{stmaryrd}
\usepackage{mathrsfs}
\usepackage{caption}

\newcommand{\cU}{\mathcal{U}}
\newcommand{\bZ}{\mathbb{Z}}
\newcommand{\bG}{\mathbb{G}}
\newcommand{\cO}{\mathcal{O}}

\newcommand{\bP}{\mathbb{P}}

\newcommand{\bA}{\mathbb{A}}

\newcommand{\spec}{\operatorname{Spec}}

\newcommand{\Gr}{\operatorname{Gr}}
\newcommand{\Pic}{\operatorname{Pic}}
\newcommand{\diag}{\operatorname{diag}}

\newcommand{\GL}{\operatorname{GL}}
\newcommand{\Hilb}{\operatorname{Hilb}}
\newcommand{\Aut}{\operatorname{Aut}}
\newcommand{\PGL}{\operatorname{PGL}}

\newcommand{\Id}{\operatorname{Id}}

\newcommand{\sgn}{\operatorname{sgn}}

\newtheorem{theorem}{Theorem}[section]
\newtheorem{Def}[theorem]{Definition}
\newtheorem{Teo}[theorem]{Theorem}
\newtheorem*{TeoA}{Theorem A.5}
\newtheorem{Lemma}[theorem]{Lemma}
\newtheorem{Oss}[theorem]{Observation}

\newtheorem{Prop}[theorem]{Proposition}

\newtheorem{notations}[theorem]{Notation}

\newtheorem{Rmk}[theorem]{Remark}

\begin{document}
\title{The Picard Group of the Moduli of Smooth Complete Intersections of Two Quadrics}
\author{Shamil Asgarli and Giovanni Inchiostro}
\begin{abstract} We study the moduli space of smooth complete intersections of two quadrics in $\mathbb{P}^n$ by relating it to the geometry of the singular members of the corresponding pencils. By giving an alternative presentation for the moduli space of complete intersections, we compute the Picard group for all $n\geq 3$.
\end{abstract}

\maketitle 

\section{Introduction}\label{Section:introduction}

Given a scheme $X$ and an algebraic group $G$ that acts on $X$, one might be interested in studying the Chow ring $A^{*}([X/G])$. The notion of integral Chow ring for a smooth quotient stack was introduced by Edidin-Graham in \cite{EdGr}, where they developed the general framework of equivariant intersection theory. Some explicit examples have been computed in \cite{VisM2}, \cite{EdFulNodal}, and \cite{EdFulHyper}. More recently, in \cite{FV} the authors treat the case when $X$ is the moduli space of smooth hypersurfaces of degree $d$ in $\bP^{n}$, and $G=\GL_{n+1}$. In particular, they completely determine the Chow ring of the space of smooth plane cubics. 

A natural case to consider is the stack of complete intersections of two quadrics in $\bP^{n}$. With a view toward understanding the full equivariant Chow ring, we determine the first graded piece, namely the Picard group:
\begin{Teo}\label{Teo:answer:pic} We have 
$$
\Pic \left(\left\{\begin{matrix} \text{Complete intersections} \\ \text{of two quadrics in }\mathbb{P}^n
\end{matrix} \right\}\right ) \cong
\begin{cases}
\bZ/n\bZ &\text{if } n \text{ is even} \\
\bZ/2n\bZ &\text{if } n\equiv 1 \text{ (mod } 4) \\
\bZ/4n\bZ &\text{if } n\equiv 3 \text{ (mod } 4)
\end{cases} 
$$
\end{Teo}

An important ingredient in our work is the connection between the space of complete intersections of two quadrics in $\mathbb{P}^n$ and the binary forms of degree $n+1$. Namely, one can associate to the complete intersection $X=Q_1 \cap Q_2$, the binary form $\det(x_0Q_1+x_1Q_2)$. 
It turns out that $X$ is smooth if and only if $\det(x_0Q_1+x_1Q_2)$ has distinct roots (see \cite{Reid}).
This construction can be 
made functorial, leading to a morphism between moduli spaces 
$$\Phi:\left\{\begin{matrix} \text{Complete intersections} \\ \text{of two quadrics in }\mathbb{P}^n
\end{matrix} \right \} \to \left\{\begin{matrix} \text{Binary forms of degree }n+1 \\ \text{with distinct roots}
\end{matrix} \right\}$$

It was already known to A. Cauchy and C. Jacobi that if $Q_1 \cap Q_2$ is smooth, then there is a basis in which $Q_1$ and $Q_2$ are simultaneously diagonal, i.e. $Q_1= \sum_{i=0}^n a_{i}x_i^2$ and $Q_2= \sum _{i=0}^n b_{i}x_i^2$. These coefficients can be packaged into the space $W \subseteq \mathbb{A}^{2n+2}$:
$$
W:=\left\{\begin{pmatrix} a_{0} & a_{1} &\cdots &a_{n}\\ b_{0} & b_{1} &\cdots & b_{n} 
\end{pmatrix}: \text{for every } i \neq j, \ a_{i}b_{j} \neq
 a_{j}b_{i} \right\}
$$
where the condition $\ a_{i}b_{j} \neq
 a_{j}b_{i}$ is imposed by the smoothness of $Q_1 \cap Q_2$ (Lemma \ref{Singular:complete:intersection}).
We refer to $W$ as the \textbf{diagonal slice}. 

For each non-zero integer $k$ we define groups $G_k$ which act on the diagonal slice. These groups will be quotients of $(\GL_2 \times \mathbb{G}_m^{n+1}) \rtimes S_{n+1}$ (see Section \ref{section:diagonal:slice}). For $a|b$ there is a homomorphism $G_b \to G_a$ which makes the identity map $W\to W$ equivariant. This gives rise to a morphism $[W/G_{b}] \to [W/G_{a}]$. The relevance of this construction lies in Theorem \ref{Teo:introduction:diagram}: 

\begin{Teo}\label{Teo:introduction:diagram}There are isomorphisms
$$ [W/G_{-2}] \cong \left\{\begin{matrix} \text{Complete intersections} \\ \text{of two quadrics in }\mathbb{P}^n
\end{matrix} \right \} \text{ and } \text{ } [W/G_{-1}] \cong \left\{\begin{matrix} \text{Binary forms of degree }n+1 \\ \text{with distinct roots}
\end{matrix} \right\}$$
such that the following diagram commutes:
$$
\xymatrix{
{\left\{\begin{matrix} \text{Complete intersections} \\ \text{of two quadrics in }\mathbb{P}^n
\end{matrix} \right \}} \ar[r]^-{\Phi} & {\left\{\begin{matrix} \text{Binary forms of degree }n+1 \\ \text{with distinct roots}
\end{matrix} \right\}} \\
[W/G_{-2}] \ar[r] \ar[u] & [W/G_{-1}] \ar[u]
}
$$
\end{Teo}
\noindent Two advantages of using this alternative presentation are: \\
$\bullet$ $W$ is affine, which is easier to handle; we can compute 
$\Pic([W/G_{-k}])$ for every $k$ (see Theorem \ref{Theo:Pic}).  \\
$\bullet$ The map induced by $\Phi$ on Picard groups is readily available (see Proposition \ref{prop:pic:map}).

There are connections between complete intersections of quadrics and hyperelliptic curves. Indeed, given a smooth complete intersection $X=Q_1\cap Q_2$ in $\bP^{2g+1}$, consider the hyperelliptic curve given by the equation $y^2=\det(xQ_1+zQ_2)$ in $\bP(1, 1, g+1)$. Carrying out this construction in families, we obtain a factorization of the map $\Phi$ (Section \ref{section:hyperelliptic}):
$$\left\{\begin{matrix} \text{Complete intersections} \\ \text{of two quadrics in $\bP^{2g+1}$}
\end{matrix} \right\} \to \left\{\begin{matrix} \text{Hyperelliptic curves} \\ \text{of genus $g$} \end{matrix}\right\} \to \left\{\begin{matrix} \text{Binary forms of degree } 2g+2 \\ \text{with distinct roots} \end{matrix} \right\}
$$
Combining our analysis with the results in Gorchinskiy-Viviani \cite{GorViv}, we prove:
\begin{Teo}\label{pic:hyper} The induced map
$$ 
\Pic\left(\left\{\begin{matrix} \text{Hyperelliptic curves} \\ \text{of genus $g$} \end{matrix}\right\}\right) \to \Pic \left(\left\{\begin{matrix} \text{Complete intersections} \\ \text{of two quadrics in $\bP^{2g+1}$}
\end{matrix} \right\}\right)
$$
is an isomorphism. 
\end{Teo}
More generally, the stack of uniform cyclic covers of degree $n$ over a smooth curve, and its Picard group have been studied by Poma, Talpo, and Tonini in \cite{PTT}. 

The present paper is organized as follows. In Section \ref{section:context} we provide background and historical context about complete intersections of two quadrics, and introduce the relevant moduli spaces. The groups $G_k$ together with their actions on the diagonal slice are introduced in Section \ref{section:diagonal:slice}. Both Section \ref{section:main:diag} and Section \ref{section:isom} are devoted to the proof of Theorem \ref{Teo:introduction:diagram}. The commutativity of the diagram is checked in Section \ref{section:main:diag}, and the two main isomorphisms are proved in Section \ref{section:isom}. In Section \ref{section:picard:group} we address the problem of computing the Picard groups and the relevant map between them; in particular, we prove Theorem \ref{Teo:answer:pic}. In Section \ref{section:hyperelliptic} we revisit the connection with the hyperelliptic curves, and prove Theorem \ref{pic:hyper}. Finally, in the appendix, a criterion is given for checking when a map between two normal Artin stacks is an isomorphism.

\begin{bf}Conventions:\end{bf} We will work over an algebraically closed field $k$ of characteristic 0. Throughout the paper we will assume that $n\ge 3$, and for simplicity, we will let $N=\binom{n+2}{2}$ be the shorthand for $\dim (H^{0}(\mathcal{O}_{\bP^n}(2)))$. When we say that a diagram of stacks commutes, we mean that it is 2-commutative. 

\begin{bf}Acknowledgements.\end{bf} We thank our advisors Dan Abramovich and Brendan Hassett for their constant support and many helpful discussions. We are also grateful for insightful conversations with Kenneth Ascher, Asher Auel, Dori Bejleri, Damiano Fulghesu, Dhruv Ranganathan, Mattia Talpo, Angelo Vistoli and Yuwei Zhu. We thank the referee for the detailed comments on the manuscript. Research by the authors are partially supported by funds from the NSF grants DMS-1551514 and DMS-1500525, respectively.

\section{Context and Background}\label{section:context}

The complete intersections of two quadrics in $\bP^n$ have been intensively studied since the time of Pl\"ucker, Kummer and Klein \cite{Klein}. When $n=3$, the smooth complete intersections of two quadrics are exactly the genus 1 curves. Given such a complete intersection $C=Q_1 \cap Q_2$, we can consider the pencil of quadrics $\{s Q_1 + t Q_2\} \to \bP^1$. This pencil has exactly four singular members, and let $p_1,...,p_4$ be the corresponding points in $\bP^1$. With more work we can see that the degree 2 covering of $\bP^1$ ramified at the points $p_i$ is isomorphic to $C$.

When $n=4$, the resulting objects are Del Pezzo surfaces. Given a smooth Del Pezzo surface $X=Q_1\cap Q_2$, there are exactly $5$ points in $\mathbb{P}^1$ that correspond to singular members of the pencil $\{s Q_1 + t Q_2\} \to \bP^1$. If we embed $\mathbb{P}^1 \hookrightarrow \mathbb{P}^2$ through the Veronese embedding, the blow-up of $\bP^2$ at these five points is isomorphic to $X$. This classical result was generalized by Skorobogatov in positive characteristic \cite{Skor}.

In his thesis \cite{Reid}*{Proposition 2.1} Reid shows that the isomorphism class of $X$ is uniquely determined by the configuration of points in $\bP^1$ corresponding to singular members of the family $\{s Q_1 + t Q_2\} \to \bP^1$. 

These examples indicate how the configuration of the singular members of the pencil and the original complete intersection are related. Such a connection will play an essential role in our work. 
This connection, and the corresponding moduli interpretation, is already taken into account by Hassett-Kresch-Tschinkel in \cite{HKT}, where they also present a compactification of the space of smooth degree 4 Del Pezzo surfaces. Other compactifications are considered by Mabuchi-Mukai \cite{MabMuk} and Hacking-Keel-Tevelev \cite{HaKeTe}.

\noindent To formalize our discussion, we need to properly introduce the relevant moduli spaces. 

\textbf{Notation.} Throughout the paper, we will adopt the following notation: \\
$\bullet$ $\mathcal{U}\subseteq\Gr(2, H^0(\mathcal{O}_{\bP^n}(2)))$ will denote the open subset parametrizing smooth complete intersections of two quadrics. \\
$\bullet$ $U^{\circ}\subseteq\bP(H^0(\mathcal{O}_{\bP^1}(n+1)))$ will denote the open subset consisting of binary forms of degree $n+1$ with distinct roots. 

\textbf{Group actions.} On the spaces mentioned above we will consider the following actions: \\
$\bullet$ $\PGL_{n+1}$ acts on $\bP^{n}$ by change of coordinates, and preserves the property of being a complete intersection. More precisely, given $[A]\in\PGL_{n+1}$, $[A]$ acts on $V=\langle Q_1, Q_2\rangle\in\cU$ by $$[A]\ast V := \langle (A^{-1})^{T} Q_1 A^{-1}, (A^{-1})^{T} Q_2 A^{-1}\rangle $$ 
$\bullet$ Given $p(x, y)\in U^{\circ}$ and $M=\begin{pmatrix} a & b \\ c & d \end{pmatrix}$, we define $[M] \ast p(x, y) := p(ax+cy, bx+dy)$.

\textbf{Moduli.} The moduli spaces discussed in the introduction can be understood as follows:
\begin{align*}
\left\{\begin{matrix} \text{Complete intersections} \\ \text{of two quadrics in }\mathbb{P}^n
\end{matrix} \right \} = [\mathcal{U}/\PGL_{n+1}]
\end{align*}
and 
\begin{align*}
\left\{\begin{matrix} \text{Binary forms of degree } n+1 \\ \text{with distinct roots}
\end{matrix} \right\} = [U^{\circ}/\PGL_{2}] 
\end{align*}

Let us take a closer look at the functorial properties of $[\mathcal{U}/\PGL_{n+1}]$. If $X_1$ and $X_{2}$ are complete intersections of two quadrics in $\bP^{n}$, they have the same Hilbert polynomial $p(t)$. In particular, if $V=\langle Q_1, Q_2\rangle \subseteq H^{0}(\mathcal{O}_{\bP^{n}}(2))$ is such that $Q_1\cap Q_2$ is a \textit{complete} intersection, then $Q_{1}\cap Q_{2}$ corresponds to a point of $\Hilb_{\bP^{n}}^{p(t)}$. In fact, this set-theoretic function can be turned into a rational map $\Gr(2, N) \dashrightarrow \Hilb_{\bP^{n}}^{p(t)}$. In \cite{AV}, Avritzer-Vainsencher perform a sequence of blow-ups $G'\to \Gr(2, N)$ away from $\mathcal{U}$ to resolve the map $\Gr(2, N) \dashrightarrow \Hilb_{\bP^{n}}^{p(t)}$. They prove that $G'\dashrightarrow \Hilb_{\bP^{n}}^{p(t)}$ is an isomorphism. Consequently, $\mathcal{U}$ is an open subscheme of the Hilbert scheme; this reinforces the moduli interpretation of $[\mathcal{U}/\PGL_{n+1}]$.

\section{The Diagonal Slice}\label{section:diagonal:slice}

In this section we interpret the diagonal slice $W$ geometrically, and define the groups $G_k$ acting on it.
Recall that we defined $W$ to be the set of 2$\times (n+1)$ matrices, where all the $2\times$2 determinants are not 0. We can view it as the total space of a particular frame bundle.

\textbf{Notation.} Let $\mathcal{F}$ be the frame bundle associated to the universal subbundle of $\Gr(2, N)$ restricted to $\mathcal{U}$. 

Consider $V:=\langle x_0^2, ..., x_n^2\rangle\subseteq H^0(\mathcal{O}_{\bP^n}(2))$. The inclusion $V\hookrightarrow H^0(\mathcal{O}_{\bP^n}(2))$ induces an inclusion $\Gr(2, n+1)\hookrightarrow\Gr(2, N)$. Let then $\mathcal{U}' := \mathcal{U}\cap \Gr(2,n+1)$.

\begin{Oss}\label{obs:frame} The diagonal slice $W$ is isomorphic to the restriction of $\mathcal{F}$ to $\mathcal{U'}$. \end{Oss}

Indeed, a point in $\mathcal{F}' := \mathcal{F}|_{\mathcal{U}'}$ can be represented by a pair $(Q_1, Q_2)$ of diagonal quadrics such that $Q_1\cap Q_2$ is smooth. Using Lemma \ref{Singular:complete:intersection} below, we see that the map $W\to \mathcal{F}'$ sending
$$
\begin{pmatrix} a_{0} & a_{1} &\cdots &a_{n}\\ b_{0} & b_{1} &\cdots &b_{n} 
\end{pmatrix}\mapsto \left(\sum_{i=0}^{n} a_{i} x_i^2, \sum_{i=0}^{n} b_{i} x_i^2\right)
$$
is an isomorphism.

\begin{Lemma}\label{Singular:complete:intersection}
Let $X \subseteq \mathbb{P}^n$ 
be the complete intersection of two quadrics given by the matrices $\diag(\alpha_0,...,\alpha_n)$, $\diag(\beta_0,....,\beta_n)$.

Then $X$ is singular if and only if there are $0 \le i < j \le n$ such that $\alpha_i\beta_j=\alpha_j \beta_i$.
\end{Lemma}

The proof of the lemma follows from the Jacobian criterion for smoothness.

Now we will construct the groups $G_k$. 
First, consider the action of $S_{n+1}$ on $\GL_2 \times \mathbb{G}_m^{n+1}$ given by permuting the $n+1$ factors: $$\sigma *(M,(\lambda_0,...,\lambda_n)):=
(M,(\lambda_{\sigma^{-1}(0)},...,\lambda_{\sigma^{-1}(n)}))$$

This action respects the group multiplication in
$\GL_{2} \times \mathbb{G}_m^{n+1}$, therefore gives a homomorphism $\rho: S_{n+1} \to \operatorname{Aut}(\GL_{2} \times \mathbb{G}_m^{n+1})$.
\begin{Def}
Let $\mathcal{G}:=(\GL_{2} \times \mathbb{G}_m^{n+1})\rtimes_{\rho} S_{n+1}$.\end{Def} For every non-zero integer $k$, let $\mathcal{N}_k:=\langle (\diag(\lambda^{-k}, \lambda^{-k}),(\lambda,...,\lambda), \Id)\rangle \subseteq \mathcal{G}$. This is a normal subgroup. We now introduce the main groups of this paper:
\begin{Def}Let $G_k:=\mathcal{G}/\mathcal{N}_k$.\end{Def}
\begin{Oss}
 Whenever $a\mid b$, we have a homomorphism $h_{b,a}:G_{b} \to G_{a}$, which is surjective with kernel of order $(\frac{b}{a})^n$.
\end{Oss}
In fact, we can consider the homomorphism $\mathfrak{h}_{b,a}:\mathcal{G} \to \mathcal{G}$ defined by
$$(M, (\lambda_i)_{i=0}^n,\sigma) \mapsto (M, (\lambda_i^{\frac{b}{a}})_{i=0}^n,\sigma)$$ It sends
$(\diag(\lambda^{-b},\lambda^{-b}),(\lambda,...,\lambda),\sigma)\mapsto
(\diag(\lambda^{-b},\lambda^{-b}),(\lambda^{\frac{b}{a}},...,\lambda^\frac{b}{a}),\sigma)$, namely $\mathcal{N}_b$
goes
to $\mathcal{N}_a$. As a result it induces a homomorphism $h_{b,a}:G_b \to G_a$. 

Let $k$ be an integer.
We have an action of $\mathbb{G}_m^{n+1}$ on $W$: $$ (\lambda_0,...,\lambda_n) *
\begin{pmatrix} a_{0} & a_{1} &\cdots &a_{n}\\ b_{0} & b_{1} &\cdots &b_{n} 
\end{pmatrix} = \begin{pmatrix} \lambda_0^k a_{0} & \lambda_1^k a_{1} &\cdots & \lambda_n^k a_{n}\\ \lambda_0^k b_{0} &
\lambda_1^k b_{1} &\cdots &\lambda_{n}^{k} b_{n} 
\end{pmatrix}$$ 
We also have an action of $\GL_2$ on the diagonal slice by left multiplication. These
two actions combine together to give an action of $\GL_2 \times \mathbb{G}_m^{n+1}$ on $W$.

Finally we have an action of $S_{n+1}$ on the diagonal slice:
 $$ \sigma*
\begin{pmatrix} a_{0} & a_{1} &\cdots &a_{n}\\ b_{0} & b_{1} &\cdots & b_{n} 
\end{pmatrix} = \begin{pmatrix} a_{\sigma^{-1}(0)} &  a_{\sigma^{-1}(1)} &\cdots &  a_{\sigma^{-1}(n)}\\ b_{\sigma^{-1}(0)} & b_{\sigma^{-1}(1)} &\cdots & b_{\sigma^{-1}(n)} 
\end{pmatrix}$$

To define an action $*_k$ of $G_k$ on $W$, recall the following lemma:
\begin{Lemma}\label{Lemma:semidirect} Suppose that $H$ and $K$ are two groups acting on a set $X$.
Let $\varphi: K\to \Aut(H)$ be a homomorphism, and let $\mathcal{G}=H\rtimes_{\varphi} K$. Assume that
$k\cdot (h\cdot (k^{-1}\cdot x)) = \varphi_{k}(h)\cdot x$ for every $x\in X$, $h\in H$, $k\in K$.

Then there is an action of $\mathcal{G}$ on $X$ via $g\cdot x = (hk)\cdot x := h\cdot(k\cdot x)$.\end{Lemma}

In this case, we have $H=\GL_{2}\times\mathbb{G}_{m}^{n+1}$ and $K=S_{n+1}$. One can check that
$$
\sigma \cdot (h \cdot (\sigma^{-1} \cdot x)) = \rho_{\sigma}(h)\cdot x
$$
for every $\sigma \in S_{n+1}$ and $h \in \GL_2 \times \mathbb{G}_m^{n+1}$. Since $\mathcal{N}_k$ acts trivially on $W$, this induces an action of $G_k$ on the diagonal slice.

\begin{Oss}\label{Oss:G-2frame}
Identifying $W$ with $\mathcal{F}'$ as in Observation \ref{obs:frame}, the action of $\GL_2 \subseteq G_{-2}$ on $W$ coincides with the action of $\GL_2$ on the fibers of the $\GL_2$-bundle $\mathcal{F}' \to \mathcal{U}'$.  
\end{Oss}
If $a$ and $b$ are integers such that $a|b$, then the homomorphism $h_{b,a}$ induces a morphism of quotient stacks $F_{b,a}:[W/G_b] \to [W/G_a]$.

\begin{Oss}\label{Obs:stabilizer:gerbe} Let $w$ be a point of $W$. Since $h_{b,a}$ is surjective with kernel of size $(\frac{b}{a})^n$, the induced homomorphism $\operatorname{Stab}_{G_b}(w) \to \operatorname{Stab}_{G_a}(w)$ is also surjective, and has kernel of size $(\frac{b}{a})^n$. Thus, $|\operatorname{Stab}_{G_b}(w)|=\left(\frac{b}{a}\right)^n |\operatorname{Stab}_{G_a}(w)|$.
\end{Oss}

\section{Main Diagram}\label{section:main:diag}

In this section we describe all the arrows that constitute the \textit{main diagram}:

$$
\xymatrix{
[\mathcal{U}/\PGL_{n+1}] \ar[r]^{}  & [U^\circ/\PGL_{2}]  \\
[W/G_{-2}] \ar[u]^{\cong} \ar[r] & [W/G_{-1}]  \ar[u]_{\cong} 
}
$$

This section proceeds as follows: 

(A) We take a closer look at the map $[\mathcal{U}/\PGL_{n+1}] \to [U^{\circ}/\PGL_{2}]$. 

(B) We construct the morphism $[W/G_{-1}]\to [U^{\circ}/\PGL_{2}]$. 

(C) We construct the morphism $[W/G_{-2}]\to [\mathcal{U}/\PGL_{n+1}]$.

(D) We check that the main diagram is commutative. 

$\text{ }$
\\(A) Let $\xi':\mathcal{F} \to \mathbb{P}^{n+1}=\mathbb{P}(H^0(\mathcal{O}_{\mathbb{P}^1}(n+1)))$,
 defined by sending $$(V, v_1,v_2) \mapsto \det(x_0 v_1 + x_1 v_2)$$
 From \cite{Reid}*{Proposition 2.1}, it lands in $U^{\circ}$.
 We have an action of $\GL_2$ on the fibers of $\mathcal{F}$:
 $$\left(\begin{pmatrix}
a &b\\
c&d\\
\end{pmatrix} , (V,v_1,v_2)\right) \mapsto (V,av_1+bv_2,cv_1+dv_2)$$
 such that $\mathcal{F} \to \mathcal{U}$ is a principal $\GL_2$-bundle (as it is a frame bundle).
 
 \begin{Def}
  Let $\mathcal{F}^{\circ}:=\mathcal{F}/\mathbb{G}_m$, where $\mathbb{G}_m$ consists of the multiples of $\Id$ in $\GL_2$. This makes $\mathcal{F}^{\circ} \to \mathcal{U}$ a principal $\PGL_2$-bundle. Then $\xi'$ induces $\xi: \mathcal{F}^{\circ} \to U^{\circ}$.
 \end{Def}

 This allows us to define a natural action of $\PGL_{n+1}$ on $\mathcal{F}^{\circ}$, given by $$A*(V,v_1,v_2):=((A^{-1})^{T}VA^{-1},
 (A^{-1})^{T}v_1A^{-1},(A^{-1})^{T}v_2A^{-1})$$ This commutes with the canonical action of $\PGL_2$ on the fibers. As a result we have an action of $\PGL_{2} \times \PGL_{n+1}$ on $\mathcal{F}^{\circ}$.
 
  \begin{Lemma}The first projection $\PGL_{2} \times \PGL_{n+1}  \to \PGL_{2}$ makes $\xi$ equivariant.\end{Lemma}
 \begin{proof}
 We have to check:
$$
\xi((M, A)\cdot (V, [v_1, v_2])) = M\cdot \xi(V, [v_1, v_2])
$$
for any $A\in\PGL_{n+1}$ and $M\in\PGL_{2}$. Let $M = \begin{pmatrix} a & b \\ c & d \end{pmatrix}$. We have:
\begin{align*}
\xi((M, A)\cdot (V, [v_1, v_2])) &= \det(x_0(A^{-1})^T(av_1+bv_2)A^{-1}+x_1 (A^{-1})^{T}(cv_1+dv_2) A^{-1}) \\
&=\det(A^{-1})^2\cdot \det(x_0(av_1+bv_2)+x_1(cv_1+dv_2))
\end{align*}
On the other hand,
$$
M\cdot \xi(V, [v_1, v_2]) = M\cdot \det(x_0 v_1 + x_1 v_2)=\det((ax_0+cx_1)v_1+(bx_0+dx_1)v_2)
$$
which equals $\det(x_0(av_1+bv_2)+x_1(cv_1+dv_2))$.
Since these forms differ by a scalar (namely $\det(A^{-1})^2$), they represent the same element of $U^{\circ}$. Thus, $\xi$ is equivariant.
 \end{proof}
 This induces a morphism $[\mathcal{F}^{\circ}/\PGL_{2} \times \PGL_{n+1}] \to [U^{\circ}/\PGL_2]$.
 In order to relate the stack $[\mathcal{F}^{\circ}/\PGL_{2} \times \PGL_{n+1}]$ to $[\cU/\PGL_{n+1}]$, we use the following well-known result:
 
 \begin{Lemma}[\cite{Rom}]\label{Lemma:double:quot}
 Let $G$ be an algebraic group, with a normal subgroup $H$. Assume that $G$ acts on a scheme $X$, and assume that $[X/H]$ is a an algebraic space.
 
 Then $G/H$ acts on $[X/H]$, and $[X/G] \cong [[X/H]/(G/H)]$.
\end{Lemma}
 
 Since $\mathcal{F}^{\circ} \to \mathcal{U}$ is a principal $\PGL_2$-bundle, from Lemma
 \ref{Lemma:double:quot} 
 \begin{align*}
 [\mathcal{F}^{\circ}/\PGL_{2} \times \PGL_{n+1}] &\cong [[\mathcal{F}^{\circ}/\PGL_{2}]/(\PGL_{2}\times\PGL_{n+1}/\PGL_{2})]  \\
 &\cong [\mathcal{U}/\PGL_{n+1}]
 \end{align*}
 Then we get a map
 $$\Phi:[\mathcal{U}/\PGL_{n+1}] \to [U^{\circ}/\PGL_2]$$
which is the one we discussed in the introduction.
$\text{ }$\\
(B) Given a point  $\begin{pmatrix} a_{0} & a_{1} &\cdots & a_{n} \\ b_{0} & b_{1} &\cdots & b_{n} 
\end{pmatrix} \in W$, we can associate to it the binary form $\prod_{i=0}^{n}(a_{i} x_0 + b_{i} x_1)$. This gives rise to a map $\theta: W \to U^{\circ}$. In order to produce a map between the corresponding quotient stacks, we produce a group homomorphism $G_{-1}\to\PGL_{2}$ which makes $\theta$ equivariant. We send $(A, (\lambda_{i})_{i=0}^{n}, \sigma)$ to $[A]\in\PGL_{2}$. It is immediate to check that $\theta$ is equivariant 
and let $\Theta:[W/G_{-1}] \to [U^\circ/\PGL_2]$ be the induced map.
$\text{ }$ \\
(C) Given a point $p:=\begin{pmatrix} a_{0} & a_{1} &\cdots &a_{n}\\ b_{0} & b_{1} &\cdots & b_{n} 
\end{pmatrix} \in W$, let $v_{1}=\diag(a_{0},...,a_{n})$, $v_{2}=\diag(b_{0},..., b_{n})$ and $V=\langle v_1, v_2\rangle\in\Gr(2, N)$. By Lemma \ref{Singular:complete:intersection}, it follows that $V\in \mathcal{U}$. Thus, we obtain a map 
$
\widetilde{f}: W \to \mathcal{F}, \text{ }
p\mapsto (V, v_1, v_2)$. Composing with the projection $\mathcal{F} \to [\mathcal{F}/\mathbb{G}_m]=\mathcal{F}^{\circ}$, we get
$$f:W \to \mathcal{F}^{\circ}$$ 
To induce a map between the quotient stacks, we produce a group homomorphism $\psi: G_{-2} \to \PGL_{2}\times \PGL_{n+1}$ compatible with $f$. First, consider $\mathcal{G}=\GL_{2}\times(\bG_{m}^{n+1}\rtimes S_{n+1}) \overset{ \pi_{2, 3} }{\longrightarrow} \bG_{m}^{n+1}\rtimes S_{n+1}$. Identifying $\bG_{m}^{n+1}\rtimes S_{n+1}\cong N_{\GL_{n+1}}(T)$ where the latter is the normalizer of the maximal diagonal torus $T\subseteq\GL_{n+1}$, we get $\mathcal{G}\to\GL_{n+1}$. Composing with the projection $\GL_{n+1}\to\PGL_{n+1}$ induces a map $G_{-2}\to\PGL_{n+1}$. Coupling it with $G_{-2}\to\PGL_{2}$, $(M, (\lambda_{i}), \sigma)\mapsto [M]$, we finally get a homomorphism $\psi: G_{-2}\to \PGL_{2}\times\PGL_{n+1}$. 

To specify this map concretely, one needs to choose an isomorphism $\bG_{m}^{n+1}\rtimes S_{n+1} \cong N_{\GL_{n+1}}(T)$. We pick $((\lambda_{i})_{i=0}^{n}, \sigma) \mapsto \diag(\lambda_0, ..., \lambda_n) A_{\sigma}$, where $A_{\sigma}$ is the matrix sending $e_{i}\mapsto e_{\sigma(i)}$. 

\begin{Lemma}
 $f$ is equivariant with respect to $\psi$.
\end{Lemma}
\begin{proof}
The assertion is that
$$
f\left( (M, (\lambda_{i})_{i=0}^{n}, \sigma)\cdot \begin{pmatrix} a_{0} &\cdots &a_{n} \\ b_{0} & \cdots & b_{n} \end{pmatrix}\right) = \psi((M, (\lambda_{i})_{i=0}^{n}, \sigma)) \cdot 
f\left(\begin{pmatrix} a_{0} & \cdots &a_{n}\\ b_{0}  &\cdots & b_{n} 
\end{pmatrix}\right)
$$
for every $\begin{pmatrix} a_{0} &\cdots &a_{n}\\ b_{0} &\cdots & b_{n} 
\end{pmatrix} \in W$ and $(M, (\lambda_{i})_{i=0}^{n}, \sigma)\in G_{-2}$. Since every element $(M,(\lambda_i)_{i=0}^n,\sigma) \in G_{-2}$ is a product of elements of the form $(\Id, (\lambda_i)_{i=0}^n, \sigma)$ and $(M, (1,..,1), \Id)$, we can just check these two cases.

\textbf{Case 1.} $(M, (\lambda_{i})_{i=0}^{n}, \sigma) = (\Id, (\lambda_{i})_{i=0}^{n}, \sigma)$.  \\
We have
\begin{align*}
&f\left( (\Id, (\lambda_{i})_{i=0}^{n}, \sigma)\cdot \begin{pmatrix} a_{0} & \cdots &a_{n}\\ b_{0} & \cdots &b_{n} \end{pmatrix}\right) = 
f\left(\begin{pmatrix} \lambda_0^{-2} a_{\sigma^{-1}(0)} &\cdots & \lambda_{n}^{-2} a_{\sigma^{-1}(n)}\\ \lambda_{0}^{-2} b_{\sigma^{-1}(0)} &\cdots & \lambda_{n}^{-2} b_{\sigma^{-1}(n)} \end{pmatrix} \right)\\& =(V, \diag(\lambda_{0}^{-2} a_{\sigma^{-1}(0)}, ..., \lambda_{n}^{-2} a_{\sigma^{-1}(n)}), 
\diag(\lambda_{0}^{-2} b_{\sigma^{-1}(0)}, ..., \lambda_{n}^{-2} b_{\sigma^{-1}(n)}))
\end{align*}
where $V = \langle \diag(\lambda_{0}^{-2} a_{\sigma^{-1}(0)}, ..., \lambda_{n}^{-2} a_{\sigma^{-1}(n)}), 
\diag(\lambda_{0}^{-2} b_{\sigma^{-1}(0)}, ..., \lambda_{n}^{-2} b_{\sigma^{-1}(n)}) \rangle$. On the other hand, 
\begin{align*}
\psi((\Id, (\lambda_{i})_{i=0}^{n}, \sigma)) \cdot 
f\left(\begin{pmatrix} a_{0} & \cdots &a_{n}\\ b_{0}  &\cdots &b_{n} 
\end{pmatrix}\right) 
&= \diag(\lambda_0, ..., \lambda_{n}) A_{\sigma} * f\left(\begin{pmatrix} a_{0} &\cdots &a_{n}\\ b_{0} &\cdots & b_{n}\end{pmatrix}\right) \\
&= (\tilde{V}, w_1, w_2)
\end{align*}
where $\tilde{V}=\langle w_1, w_2\rangle$ and
\begin{align*}
w_{1} &= \diag(\lambda_{0}^{-1}, ..., \lambda_{n}^{-1}) (A^{-1}_{\sigma})^{T}
\diag(a_{0}, ..., a_{n}) A_{\sigma}^{-1} \diag(\lambda_{0}^{-1}, ..., \lambda_{n}^{-1}) \\
w_{2} &= \diag(\lambda_{0}^{-1}, ..., \lambda_{n}^{-1}) (A^{-1}_{\sigma})^{T}
\diag(b_{0}, ..., b_{n}) A_{\sigma}^{-1} \diag(\lambda_{0}^{-1}, ..., \lambda_{n}^{-1}) 
\end{align*}
Now observe that
\begin{align*}
w_{1} e_{i} &= \diag(\lambda_{0}^{-1}, ..., \lambda_{n}^{-1}) (A^{-1}_{\sigma})^{T}
\diag(a_{0}, ..., a_{n}) A_{\sigma}^{-1} \diag(\lambda_{0}^{-1}, ..., \lambda_{n}^{-1}) e_{i} \\
&=\diag(\lambda_{0}^{-1}, ..., \lambda_{n}^{-1}) A_{\sigma}
\diag(a_{0}, ..., a_{n}) A_{\sigma^{-1}} \diag(\lambda_{0}^{-1}, ..., \lambda_{n}^{-1}) e_{i}  \\
&=\diag(\lambda_{0}^{-1}, ..., \lambda_{n}^{-1}) A_{\sigma}
\diag(a_{0}, ..., a_{n}) A_{\sigma^{-1}} \lambda_i^{-1} e_{i}\\
&=\diag(\lambda_{0}^{-1}, ..., \lambda_{n}^{-1}) A_{\sigma}
\diag(a_{0}, ..., a_{n}) \lambda_i^{-1} e_{\sigma^{-1}(i)}
  \\
&=\diag(\lambda_{0}^{-1}, ..., \lambda_{n}^{-1}) A_{\sigma}
a_{\sigma^{-1}(i)}\lambda_i^{-1} e_{\sigma^{-1}(i)}
 \\
&=\diag(\lambda_{0}^{-1}, ..., \lambda_{n}^{-1})
a_{\sigma^{-1}(i)}\lambda_i^{-1} e_{i}
 \\
&=\lambda_i^{-2}a_{\sigma^{-1}(i)} e_{i}
\end{align*}
Thus, $w_{1} = \diag(\lambda_{0}^{-2} a_{\sigma^{-1}(0)}, ..., \lambda_{n}^{-2} a_{\sigma^{-1}(n)})$.
Similarly, $w_{2}=\diag(\lambda_{0}^{-2} b_{\sigma^{-1}(0)}, ..., \lambda_{n}^{-2} b_{\sigma^{-1}(n)})$. This verifies the claim.

\textbf{Case 2.} $(M, (\lambda_{i})_{i=0}^{n}, \sigma) = (M, (1, 1, ..., 1), \Id)$ where $M=\begin{pmatrix} a & b \\ c & d \end{pmatrix}$. 
We have
\begin{align*}
& f\left( (M, (1, .., 1), \Id)\cdot \begin{pmatrix} a_{0} & \cdots &a_{n}\\ b_{0} &\cdots &b_{n} \end{pmatrix}\right) = 
f\left(\begin{pmatrix} a a_{0} + b b_{0} & \cdots & a a_{n} + b b_{n} \\ c a_{0} + d b_{0} &\cdots & c a_{n} + d b_{n} \end{pmatrix}\right) \\
&=(V, \diag(a a_{0} + b b_{0}, ..., a a_{n} + b b_{n}),
\diag(c a_{0} + d b_{0}, ..., c a_{n} + d b_{n}))
\end{align*}
On the other hand, 
\begin{align*}
&\psi((M, (1,...,1), \Id)) \cdot 
f\left(\begin{pmatrix} a_{0} & \cdots &a_{n}\\ b_{0} & \cdots &b_{n} 
\end{pmatrix}\right) = \begin{pmatrix}
a & b \\ 
c & d \\
\end{pmatrix}\cdot (V, \diag(a_{0}, ..., a_{n}),
\diag(b_{0}, ..., b_{n}) )\\
&=(V, a \diag(a_{0}, ..., a_{n}) + b \diag(b_{0}, ..., b_{n}),
c \diag(a_{0}, ..., a_{n}) + d \diag(b_{0}, ..., b_{n}))
\end{align*}
as desired.
\end{proof}
 
Therefore, after identifying $[\mathcal{F}^{\circ}/\PGL_{2}\times\PGL_{n+1}]$ with $[\mathcal{U}/\PGL_{n+1}]$, $f$ induces a morphism 
\begin{equation}\label{map:F}
F:[W/G_{-2}] \to [\mathcal{U}/\PGL_{n+1}]
\end{equation}

(D) Since $[\mathcal{U}/\PGL_{n+1}]\cong [\mathcal{F}^{\circ}/\PGL_{2}\times\PGL_{n+1}]$, to check that the main diagram 2-commutes, it suffices to check that the two diagrams
$$
\begin{tikzcd}
{\mathcal{F}^{\circ}} \ar[r, "\xi"] & U^{\circ} \\
W \ar[r, "\Id"'] \ar[u, "f"] & W \ar[u, "\theta"']
\end{tikzcd}
 \ \ \ \ \ \text{and} \ \ \ \  
\begin{tikzcd}
\PGL_{2}\times\PGL_{n+1} \ar[r, "\pi_1"] & \PGL_{2} \\
G_{-2} \ar[r, "h_{-2, -1}"'] \ar[u, "\psi"] & G_{-1}\ar[u]
\end{tikzcd}
$$
commute. This is indeed the case for the diagram on the left: 
$$
\xymatrix{
(\diag(a_{0},...,a_{n}),\diag(b_{0},...,b_{n})) \ar[r]^-{\xi} & \displaystyle{\prod_{i=0}^n(x_0a_{i}+x_1b_{i})} \\
{\begin{pmatrix} a_{0} & \cdots &a_{n}\\ b_{0} & \cdots &b_{n} 
\end{pmatrix}}  \ar[r]_-{\Id} \ar[u]^{f} & {\begin{pmatrix} a_{0} & \cdots &a_{n}\\ b_{0} & \cdots &b_{n} 
\end{pmatrix}} \ar[u]_{\theta}
}
$$
It is easy to check the commutativity for the diagram on the right.

\section{Main Isomorphisms}\label{section:isom}
The goal of this section is to prove Theorem \ref{Teo:introduction:diagram}.
\subsection{Isomorphism $[U^\circ/\PGL_2] \cong [W/G_{-1}]$:} We begin by showing $[M_{0,n+1}/ S_{n+1}] \cong [W/G_{-1}]$. Afterwards, we recall that $[M_{0,n+1}/ S_{n+1}]\cong [U^{\circ}/\PGL_{2}]$. The composition of these two isomorphisms results in a map isomorphic to $\Theta: [W/G_{-1}]\to [U^{\circ}/\PGL_{2}]$.

As a preparation, we need to be able to answer the following question:
\begin{center}
If $G$ acts on a scheme $X$, and $X$ admits a morphism $X \to Y$ whose fibers are $G$-orbits, is $X \to Y$ a principal $G$-bundle? Namely, when $Y=[X/G]$? 
\end{center}
The following well-known result goes in this direction:
\begin{Lemma}\label{Principal:bundle}
 Let $G$ be an affine group acting on a scheme $X$. Assume that $X \to Y$ is a geometric quotient, and that the action
 of $G$ on $X$ is free (see \cite{GIT}*{Definition 0.0.8}).
 
 Then $[X/G] \cong Y $.
\end{Lemma}
The properness of the action is essential; see \cite{Kol}*{Example 2.18}.
\begin{proof}
Since the action is set-theoretically free, $[X/G]$ is an algebraic space. We want to show that it is a scheme. Since the action is proper, $[X/G]$ is
separated \cite{Edidin}*{Corollary 2.2}. 
From the description of $[X/G]$, we see that there is a bijective morphism $[X/G] \to Y$, which is separated since $[X/G]$ is separated.
Now it follows from \cite{Ols}*{Theorem 7.2.10} that $[X/G]$ is a scheme. Then by definition $X \to [X/G]$ is a geometric quotient, so $[X/G] \to Y$ is an isomorphism.
\end{proof}

\begin{Prop}\label{Proposition:iso:W/G:and:M0n}
 $[M_{0,n+1}/ S_{n+1}] \cong [W/G_{-1}]$.
\end{Prop}
\begin{proof}
 Let $H:=(\GL_2 \times \mathbb{G}_m^{n+1})/\langle \diag(\lambda, \lambda),(\lambda,...,\lambda)\rangle $. Then $H$ is a normal subgroup of $G_{-1}$, and $G_{-1}/H \cong  S_{n+1}$. Then from Lemma \ref{Lemma:double:quot} it is enough to show that
 $[W/H] \cong M_{0,n+1}$, and this isomorphism has to be compatible with the actions of $G_{-1}/H$ and $S_{n+1}$.

Let $K:=\{(\operatorname{Id},(\lambda_0,...,\lambda_n)) \in H\} \cong \mathbb{G}_m^{n+1}$ and let $\widehat{W}:=\{(p_0,...,p_n) \in (\mathbb{P}^1)^{n+1}: i \neq j \Rightarrow p_i \neq p_j\}$. Note that $H/K \cong \PGL_2$. The strategy is to show that $$[W/H] \cong [[W/K]/(H/K)] \cong [\widehat{W}/\PGL_2] \cong  M_{0,n+1}$$ The first isomorphism comes from Lemma \ref{Lemma:double:quot}, whereas the isomorphisms $[W/K] \cong \widehat{W}$ and $[\widehat{W}/\PGL_2] \cong M_{0,n+1}$ remain to be shown.
To do that, we will use Lemma \ref{Principal:bundle}. The main technicality will be to prove that $[W/K]$ and $[\widehat{W}/\PGL_2]$ are separated, namely that the action of $K$ on $W$ and the one of $H/K\cong \PGL_2$ on $\widehat{W}$ are proper. For this step, we use geometric invariant theory.
  
To show that $[W/K] \cong \widehat{W}$, consider the projection morphism \begin{center}$\pi: W \to \widehat{W}$, $\begin{pmatrix} a_{0} & a_{1} &\cdots &a_{n}\\ b_{0} & b_{1} &\cdots &b_{n} 
\end{pmatrix} \mapsto ([a_{0}:b_{0}],...,[a_{n}:b_{n}])$\end{center} From \cite{GIT}*{Proposition 0.0.2}, $(\widehat{W}, \pi)$ is a geometric quotient. We show that action of $K$ on $W$ is proper. First recall that the action
of $\mathbb{G}_m$ on $\mathbb{A}^2\smallsetminus \{0\}$ by homotheties is proper. 
Then also the action of $K=(\mathbb{G}_m)^{n+1}$ on $(\mathbb{A}^2\smallsetminus \{0\})^{n+1}$ is proper, which by definition means that the map
$\Phi:K \times (\mathbb{A}^2\smallsetminus \{0\})^{n+1} \to (\mathbb{A}^2\smallsetminus \{0\})^{n+1} \times (\mathbb{A}^2\smallsetminus \{0\})^{n+1}$
which sends $(g,x) \mapsto (gx,x)$
is proper. Now, $W$ is an open subset of $(\mathbb{A}^2\smallsetminus \{0\})^{n+1}$
and since being proper is stable under base change, $\Phi^{-1}(W \times W) \to W\times W$ is proper; $W$ is $K$-invariant, so $\Phi^{-1}(W\times W)=K \times W$. Thus the action of $K$ on $W$ is proper.

 The induced action of $\PGL_2$ on $\widehat{W}$ is component by component. The space $\widehat{W}$ is the space of collections of $n+1$ distinct points in $\mathbb{P}^1$ (namely, with a particular choice of coordinates). Consider the morphism $\widehat{\pi}:\widehat{W} \to M_{0,n+1}$ which forgets the $\mathbb{P}^1$-coordinates. From \cite{GIT}*{Proposition 0.0.2} we have that $(M_{0,n+1},\widehat{\pi})$ is a geometric quotient.
 It follows from \cite{GIT}*{Chapter 3} that $n+1$ distinct points in $(\mathbb{P}^1)^{n+1}$ are stable for the linearization $L':=\cO_{(\mathbb{P}^1)^{n+1}}(2,2,....,2)$, namely $\widehat{W} \subseteq ((\mathbb{P}^1)^{n+1})^s(L')$. Then if $L:=L'_{|\widehat{W}}$, we have that $\widehat{W}^{s}(L)=\widehat{W}$, so the properness of the action follows from \cite{GIT}*{Corollary 2.2.5}.
\end{proof}

We now explain the well-known isomorphism $[M_{0,n+1}/ S_{n+1}] \cong [U^{\circ}/\PGL_{2}]$. It was proved above that $M_{0,n+1}=[\widehat{W}/\PGL_2]$. There is also an action of $S_{n+1}$ on $\widehat{W}$, which permutes the factors, and commutes with the action of $\PGL_2$. Moreover, $[\widehat{W}/S_{n+1}]\cong U^\circ$, where an isomorphism is induced by the map $\widehat{W} \to U^\circ$ sending $n+1$ distinct points to the unique binary form of degree $n+1$ vanishing on them.  Therefore: \begin{align*}[U^\circ/\PGL_{2}] &\cong [[\widehat{W}/S_{n+1}]/\PGL_2]\cong  [\widehat{W}/(S_{n+1} \times \PGL_2)]\cong \\ & \cong [[\widehat{W}/\PGL_2]/S_{n+1}]\cong [M_{0,n+1}/ S_{n+1}]\end{align*} where the second and third isomorphisms come from Lemma \ref{Lemma:double:quot}. 

\subsection{Isomorphism $[W/G_{-2}]\cong [\cU/\PGL_{n+1}]$.} Our goal is to prove the following theorem:

\begin{Teo}\label{main:theorem}
The morphism $F:[W/G_{-2}]\to [\cU/\PGL_{n+1}]$ in (\ref{map:F}) is an isomorphism.
\end{Teo}

The proof of Theorem \ref{main:theorem} relies on Theorem \ref{Proposition:iso:normal:alg:stacks} below (proved in the appendix). This is a generalization of a well-known result from scheme theory: a bijective separated morphism between two normal equidimensional schemes of finite type over an algebraically closed field of characteristic 0 is an isomorphism. This result for schemes can be seen as a corollary of Zariski's main theorem (see \cite{Qliu}*{Corollary 4.6}, or \cite{ega}*{Theorem 8.12.6, Corollary 8.12.10 and Lemma 8.12.10.1}) once we observe that such a morphism would be birational. 

Similarly, one can understand Theorem \ref{Proposition:iso:normal:alg:stacks} as a corollary of Zariski's main theorem for representable morphisms of stacks (see \cite{LMB}*{Theorem 16.5}).
\begin{TeoA}Let $k$ be an algebraically closed field of characteristic 0.
 Let $\mathcal{X}_1$, $\mathcal{X}_2$ two normal, equidimensional, Artin stacks of finite type over $k$.
 Let $f:\mathcal{X}_1 \to \mathcal{X}_2$ be a separated morphism
 such that $\mathcal{X}_1(\operatorname{Spec}(k)) \to \mathcal{X}_2(\operatorname{Spec}(k))$
 is an equivalence (i.e. fully faithful and essentially surjective).
 
 Then $f$ is an isomorphism.
\end{TeoA}

In particular, we need to check that 
$$
F : [W/G_{-2}](\spec(k)) \to [\mathcal{U}/\PGL_{n+1}](\spec(k))
$$
is fully faithful and essentially surjective. To address fully-faithfulness, one needs to understand the automorphism group of a smooth
complete intersection of two quadrics. 

Given two quadrics $Q_{1}, Q_{2}$, after a change of coordinates, we can assume
$Q_{1}=\sum a_{i} x_{i}^2$ and $Q_{2}=\sum b_{i} x_{i}^2$. In this description, we can exhibit $2^{n}$ automorphisms of 
$X=\{Q_{1}=0\}\cap \{Q_{2}=0\}$, namely $x_{i}\mapsto \pm x_{i}$. Notice that each of these automorphisms do not permute the 
singular members of the pencil $\{s Q_1 + t Q_{2}\} \mapsto \mathbb{P}^1$. Generically, these are all the automorphisms of $X$; however, there are special
complete intersections with extra automorphisms. For instance, consider 
$$
X = \left\{\sum_{i=0}^{n} x_{i}^2 = 0 \right\} \cap \left\{\sum_{i=0}^{n} \mu^{i} x_i^{2} = 0 \right\}
$$
where $\mu$ is a primitive $(n+1)$-th root of unity. In this case, an extra automorphism is given by $x_i\mapsto x_{i-1}$ for $i>0$, and
$x_{0}\mapsto x_n$. 

As this example illustrates, the extra automorphisms come from the presence of the symmetries of the coefficients. These correspond to 
the symmetries of the configuration of points in $\bP^{1}$ associated to the singular members of the pencil. From this discussion, the
automorphisms of $X$ can be divided into two sets: those fixing the singular members, and those which do not. 

Given a complete intersection $X$, it corresponds to a point $[X]\in [\mathcal{U}/\PGL_{n+1}](\spec (k))$. The content of the previous 
paragraphs is captured in this homomorphism 
$$
\Aut_{[\mathcal{U}/\PGL_{n+1}]}([X]) \overset{\Phi_{X}}{\to} \Aut_{[U^\circ/\PGL_{2}]}(\Phi([X]))
$$
Indeed, given $\psi\in \Aut_{[\mathcal{U}/\PGL_{n+1}]}([X])$, the corresponding permutation of the singular members induces the 
automorphism $\Phi_{X}(\psi)\in \Aut_{[U^\circ/\PGL_{2}]}(\Phi([X]))$. As we will show, $\ker(\Phi_{X})$ exactly consists of those
$2^{n}$ automorphisms mentioned earlier (see Lemma \ref{Lemma:cardinality:automorphism}), and $\Phi_{X}$ is surjective. These are the ingredients we will use to study  
$\Aut_{[\mathcal{U}/\PGL_{n+1}]}([X])$.

\begin{Lemma}\label{no:aut} The morphism $F:[W/G_{-2}]\to [\cU/\PGL_{n+1}]$ is representable. In particular, for each $p\in [W/G_{-2}](\spec (k))$, $\Aut_{[W/G_{-2}]}(p) \to \Aut_{[\cU/\PGL_{n+1}]}(F(p))$ is injective. 
\end{Lemma}
\begin{proof} The second statement follows from the first by \cite{stacks-project}*{Tag 04T0}, so we prove that $F$ is representable.

Using \cite{Ols}*{Exercise 10.F}, the following diagram is cartesian: 
$$\xymatrix{
[(W\times\PGL_{n+1})/G_{-2}] \ar[d] \ar[r] & [W/G_{-2}] \ar[d] \\
\mathcal{U} \ar[r] & [\mathcal{U}/\PGL_{n+1}]  
}$$
We claim that the action of $G_{-2}$ on
$W \times \PGL_{n+1}$ is free. Indeed, in part (C) of Section \ref{section:main:diag}, we introduced the map $\psi:G_{-2} \to \PGL_{2} \times \PGL_{n+1}$. Let $\psi_2:=\pi_2 \circ \psi$, where $\pi_2: \PGL_{2} \times \PGL_{n+1} \to \PGL_{n+1}$ is the second projection.
If a point $(p,g)$ is fixed by $h \in G_{-2}$, then
$h(p,g)\overset{\operatorname{def}}{=}(h*p,\psi_2(h)g)=(p,g)$. In particular, $\psi_2(h)=\Id$. This means that
there is $A \in \GL_2$ such that
$$h=(A,(\lambda,...,\lambda), \Id)$$
Since $h*p=p$, $A=\diag(\lambda^{2},\lambda^{2})$. Such an element is the identity in $G_{-2}$.

This implies that $[(W \times \PGL_{n+1})/G_{-2}]$ is an algebraic space. Since $\mathcal{U}\to [\mathcal{U}/\PGL_{n+1}]$ is smooth and surjective, \cite{stacks-project}*{Tag 04ZP} gives the desired result.
\end{proof}
\begin{Lemma}\label{Lemma:cardinality:automorphism}
 For each $p \in [\mathcal{U}/\PGL_{n+1}](\spec(k))$, the kernel of $\operatorname{Aut}(p) \xrightarrow{\Phi}
 \operatorname{Aut}(\Phi(p))$ has cardinality $2^n$ (see also \cite{Reid}*{Chapter 2}).
\end{Lemma}
\begin{proof}
 Using \cite{Ols}*{Exercise 10.F}, the following diagram is cartesian:
 $$
 \xymatrix{
 [\mathcal{F}^{\circ}\times\PGL_{2}/\PGL_{2}\times\PGL_{n+1}] \ar[r] \ar[d] & [\mathcal{F}^{\circ}/\PGL_{2}\times\PGL_{n+1}] \ar[d] \\
 U^\circ \ar[r] & [U^\circ/\PGL_{2}]
 }
 $$
Since $U^\circ \to [U^\circ / \PGL_2]$ is surjective, from \cite{Vis}*{Example 7.17} and Lemma \ref{Lemma:Atlas}
 it is enough to show that any orbit of $\mathcal{F}^{\circ} \times \PGL_2$,
 under the diagonal action of $\PGL_{2} \times \PGL_{n+1}$,
 has a point whose stabilizer has cardinality $2^n$.
 
 From \cite{Reid}*{Proposition 2.1}, we can choose a point of the form
 $$z:=(V,\diag(1,...,1),A:=\diag(\lambda_0,...,\lambda_n),g) \text{ with }g \in \PGL_2$$ with $i \neq j \Rightarrow \lambda_i \neq \lambda_j$. For every $B \in \GL_{n+1}$ let $[B]$ be its class
 in $\PGL_{n+1}$.
 Assume that $(a,[B^{-1}]) \in \PGL_2 \times \PGL_{n+1}$ stabilizes $z$,
 then $a=\Id$. Since $\mathcal{F}^{\circ}=[\mathcal{F}/\mathbb{G}_m]$ there is $c \in \mathbb{G}_m$ such that 
 $$B^{T} B=c\Id \text{ and } B^{T}AB=cA$$
 Up to changing $B$ with $\frac{1}{\sqrt{c}}B$ we can assume $c=1$: we can assume that $A$ and $B$ commute.
 Since $\lambda_i$ are all different, $B$ is diagonal: $B=\diag(b_0,...,b_n)$.
 Now from $B^{T}B=\Id$ we get $b_i=\pm 1$ for $i=0, ..., n$. However, $[-\Id]=[\Id]$ so we get $2^n$ distinct choices for $[B]$.
\end{proof}

\begin{Rmk} It is not hard to see that $\Aut_{[U^\circ/\PGL_{2}]}(q)=1$ for a generic choice of $q$. Combining this with Lemma \ref{Lemma:cardinality:automorphism}, it follows
that $\Aut_{[\mathcal{U}/\PGL_{n+1}]}(p)=2^{n}$ generically. This recovers the classical fact that a general complete intersection of two quadrics in 
$\bP^{n}$ has exactly $2^{n}$ automorphisms.\end{Rmk}

\begin{Lemma}\label{Lemma:Bij:on:orbits}
 $F$ is bijective on $k$-points.
\end{Lemma}
\begin{proof}
\underline{Surjective:} It is enough to show that every $\PGL_{n+1}$ orbit intersects the image of $f$. This follows from \cite{Reid}*{Proposition 2.1 (d)}.
 
 \underline{Injective:} Recall the main diagram of Section \ref{section:main:diag}:$$\xymatrix{
[\mathcal{U}/\PGL_{n+1}] \ar[r]^{\Phi}  & [U^\circ/\PGL_{2}]  \\
[W/G_{-2}] \ar[u]^{F} \ar[r]_{F_{-2,-1}} & [W/G_{-1}]  \ar[u]_{\Theta} 
}$$ Observe that $F_{-2, -1}$ is injective on points (since it is a gerbe), and so is $\Theta$ (as it is an isomorphism). So 
 $\Theta\circ F_{-2, -1}$ is injective on points. Thus, $\Phi\circ F$ is injective on points, implying the same conclusion
 for $F$. \end{proof}
\begin{proof}[Proof of Theorem \ref{main:theorem}] We verify the hypothesis of Theorem \ref{Proposition:iso:normal:alg:stacks}. In order to show that
$[W/G_{-2}](\spec (k)) \to [\mathcal{U}/\PGL_{n+1}](\spec (k))$ is an equivalence of categories, we will show that it is essentially surjective
and fully faithful.\\
\underline{Essential surjectivity.} This follows from Lemma \ref{Lemma:Bij:on:orbits}. \\
\underline{Fully Faithful.} It is sufficient to check that:

i) $F$ is injective on closed points. 

ii) $\Aut(p)\to\Aut(F(p))$ is bijective. \\
Lemma \ref{Lemma:Bij:on:orbits} verifies i). As $\Aut(p)\to\Aut(F(p))$ is injective by Lemma \ref{no:aut},
ii) will follow 
if we can prove that $|\Aut(p)|=|\Aut(F(p))|$. Consider the diagram
$$
\xymatrix{
1 \ar[r] & \ker(\Phi_{F(p)}) \ar[r]  & \Aut(F(p)) \ar[r]^-{\Phi_{F(p)}} & \Aut(\Phi(F(p))) &  \\
1 \ar[r] & \ker((F_{-2,-1})_{p}) \ar[r] & \Aut(p) \ar[r] \ar[u] & \Aut(F_{-2, -1}(p)) \ar[r] \ar[u] & 1 
}
$$
The bottom row is exact since $F_{-2, -1}$ is a gerbe. Using that $\Theta$ is an isomorphism and chasing the diagram,
we obtain that $\Phi_{F(p)}$ is surjective. Therefore,
\begin{align*}
|\Aut(p)| &= |\ker (F_{-2,-1})_{p}| \cdot |\Aut(F_{-2,-1}(p))| \\
\tag{Observation \ref{Obs:stabilizer:gerbe}}  &= 2^{n} \cdot |\Aut(F_{-2,-1}(p))|  \\
\tag{$\Theta$ is an isom.} &= 2^{n} \cdot |\Aut(\Theta(F_{-2,-1}(p)))| \\
&= 2^{n} \cdot |\Aut(\Phi(F(p)))| \\
\tag{Lemma \ref{Lemma:cardinality:automorphism}} &= |\ker(\Phi_{F(p)})| \cdot |\Aut(\Phi(F(p)))| \\
&=|\Aut(F(p))|
\end{align*}
Finally, observe that $F$ is separated. Indeed, $[W/G_{-1}] \cong [M_{0,n+1}/S_{n+1}]$ by Proposition \ref{Proposition:iso:W/G:and:M0n}, and $[M_{0, n+1}/S_{n+1}]$ is separated because $S_{n+1}$ is a finite group (hence acts properly on $M_{0, n+1}$). The morphism $[W/G_{-2}] \to [W/G_{-1}]$ is a gerbe which is locally trivial in the \'{e}tale topology banded by a finite group,
so this map is separated as well since it is separated \'{e}tale locally on the target.
Therefore $[W/G_{-2}]$ is separated, so $F$ is separated. Theorem \ref{Proposition:iso:normal:alg:stacks} applies.
\end{proof}

\begin{Oss}
Using Observation \ref{obs:frame}, we can identify $W$ with the frame bundle $\mathcal{F}'$. By Observation \ref{Oss:G-2frame}, $[\mathcal{F}'/\GL_{2}] \cong \mathcal{U}'$. Then 
$$
[W/G_{-2}]\cong [\mathcal{F}'/G_{-2}]
\cong [[\mathcal{F}'/\GL_{2}]/(G_{-2}/\GL_{2})] \cong [\mathcal{U}'/(G_{-2}/\GL_{2})]\cong 
[\mathcal{U}'/N_{\PGL_{n+1}}(T)]
$$
where $N_{\PGL_{n+1}}(T)$ is the normalizer of the diagonal torus in $\PGL_{n+1}$. Then we can rephrase the theorem above purely in terms of the geometry of the frame bundles. Indeed, consider the inclusion $\mathcal{U}' \overset{i}{\hookrightarrow} \mathcal{U}$. Then $\mathcal{U}'$ is $N_{\PGL_{n+1}}(T)$-invariant, namely the inclusion $N_{\PGL_{n+1}}(T)\hookrightarrow \PGL_{n+1}$ makes $i$ equivariant. Then Theorem \ref{main:theorem} becomes: the map
$$
[\mathcal{U}'/N_{T}(\PGL_{n+1})] \to [\mathcal{U}/\PGL_{n+1}]
$$
is an isomorphism.

\end{Oss}

\section{The Picard group.}\label{section:picard:group}

\subsection{Computation of the Picard group.} The goal of this section is to use Theorem \ref{main:theorem} to compute $\operatorname{Pic}([\mathcal{U}/\PGL_{n+1}])$.

The diagonal slice is described as an open subset of $\bA^{2n+2}$, and the action $G_{k}$ on $W$ extends to $\mathbb{A}^{2n+2}$. Then we have
$$A^*_{G_{k}}(\bA^{2n+2}) \to A^*_{G_{k}}(W) \to 0$$
Since the extended action of $G_{k}$ on $\bA^{2n+2}$ is \textit{linear}, this realizes $\bA^{2n+2}$ as an equivariant vector bundle over the point, so the pull-back $ A^*_{G_{k}}:=A^*_{G_{k}}(\operatorname{Spec}(k)) \to A_{G_{k}}^*(\bA^{2n+2})$ is an isomorphism.
Therefore:
$$A^*_{G_{k}} \to A^*_{G_{k}}(W) \to 0$$
and in particular $A^1_{G_{k}}(W)$ is generated by the image of $A^1_{G_{k}}$. From \cite{EdGr}*{Theorem 1},
$A^1_{G_{k}}=\operatorname{Pic}_{G_{k}}(\operatorname{Spec}(k))$ and similarly $A^1_{G_{k}}(W)=\operatorname{Pic}_{G_{k}}(W)$.
We will first understand $\operatorname{Pic}_{G_{k}}(\operatorname{Spec}(k))$, and then the kernel
of the homomorphism $A^1_{G_{k}} \to A^1_{G_{k}}(W)$.

\begin{Lemma}\label{Lemma:A1G:and:A1Gk} We have 
\begin{align*} &\text{1)} \text{ }
A^{1}_{\mathcal{G}} \cong \{(\ell, a, \delta)\in \bZ\oplus\bZ\oplus(\bZ/2\bZ)\} \\ &\text{2)} \text{ }
A^{1}_{G_{k}}\cong\{(\ell, a, \delta)\in \bZ\oplus\bZ\oplus(\bZ/2\bZ): 2 k \ell = (n+1)a\}
\end{align*}
\end{Lemma}

\begin{proof}
1) Using Theorem \cite{EdGr}*{Theorem 1} $A^{1}_{\mathcal{G}} = \Pic_{\mathcal{G}}(\spec k)$, so we can identify the former with the group
of characters of $\mathcal{G}$. Since $\mathcal{G} = (\GL_{2}\times\bG_{m}^{n+1})\rtimes_{\rho} S_{n+1}$, the characters of $\mathcal{G}$ are in
1-to-1 correspondence with triples $(\chi_1, \chi_2, \chi_3)$ of characters of $\GL_{2}, \mathbb{G}_{m}^{n+1}, S_{n+1}$ respectively, satisfying
$\chi_{3}(\sigma) \chi_{2}(\mu) \chi_{3}(\sigma^{-1}) = \chi_{2}(\rho_{\sigma}(\mu))$ for all $\mu\in\mathbb{G}_{m}^{n+1}$ and $\sigma\in S_{n+1}$. 
Recall that the characters of these three groups are of this form: $\chi_{1}: A\mapsto \det(A)^{\ell}$ for some $\ell\in\bZ$, 
$\chi_{2}: (\lambda_0, ..., \lambda_n)\mapsto \lambda_{0}^{a_0}\cdots \lambda_{n}^{a_n}$ for some $(a_{i})_{i=0}^{n}\in\bZ^{n+1}$, 
and $\chi_{3}$ is either the trivial representation or the sign representation of $S_{n+1}$. The relevant compatibility relation becomes:
$$
\sgn(\sigma) \chi_2( (\lambda_{j})_{j=0}^{n}) 
\sgn(\sigma^{-1}) = \chi_2(\rho_{\sigma} ( (\lambda_{j})_{j=0}^{n}))
$$
This is equivalent to asking:
$$
\lambda_{0}^{a_{0}}\cdots \lambda_{n}^{a_{n}}
= \lambda_{\sigma^{-1}(0)}^{a_{0}} \cdots \lambda_{\sigma^{-1}(n)}^{a_{n}}
$$
Since $\sigma$ and $\lambda_{j}$ are arbitrary, we conclude that all the exponents $a_{i}$ must be the same, and we define $a:=a_{0}=a_{1}=\cdots = a_{n}$. Therefore, we can identify a character of $\mathcal{G}$ with the corresponding triple $(\ell,a,\delta)$ where $\delta \in \mathbb{Z}/2\mathbb{Z}$.

2) Proceeding as above, $A^1_{G_{k}}=A^1_{\mathcal{G}/\mathcal{N}_{k}}=\{\chi:\mathcal{G} \to \mathbb{G}_m \ | \ \chi_{|\mathcal{N}_k} \equiv 1\}$. Thus we have: $\chi_{|\mathcal{N}_k} \equiv 1 \Leftrightarrow \chi(\diag(\lambda^{-k},\lambda^{-k}),(\lambda,...,\lambda),\Id)=1 \Leftrightarrow \chi_1(\diag(\lambda^{-k},\lambda^{-k}))\chi_2((\lambda,...,\lambda))\chi_3(\Id)=1$ which becomes $(\lambda^{-2k})^{\ell}\lambda^a \cdot ...\cdot \lambda^a =1$. Therefore we can identify $A^1_{G_{k}}$ as the triples $(\ell, a,\delta) \in \mathbb{Z} \times \mathbb{Z} \times (\mathbb{Z}/2\mathbb{Z})$ such that $2k\ell=(n+1)a$.
\end{proof}

\begin{notations}
Let $\mathcal{G} \times W \to W$ be the action which induces the action of $G_{k}$ on $W$. We will denote by $[W/_k \mathcal{G}]$ the associated quotient stack.
\end{notations}

\begin{Prop}\label{Prop:Kernel} 
\begin{align*}
\ker(A^{1}_{\mathcal{G}} \to A^{1}([W/_{k} \mathcal{G}])) &= \left\langle \left(\frac{n(n+1)}{2}, kn, 1\right)  \right\rangle
\subseteq A^{1}_{\mathcal{G}} \\
\ker(A^{1}_{G_{k}} \to A^{1}_{G_{k}}(W)) &= \left\langle \left(\frac{n(n+1)}{2}, kn, 1\right)  \right\rangle
\subseteq A^{1}_{G_{k}} 
\end{align*}
\end{Prop}

\begin{proof}
Let $\chi: \mathcal{G}\to\mathbb{G}_{m}$ be a character of $\mathcal{G}$ and let $[\mathbb{A}^{1}_{\chi}/\mathcal{G}] \to B\mathcal{G}$ be 
the corresponding line bundle. As $A^1_{\mathcal{G}}$ consists of first Chern classes of line bundles, and the first Chern class of the pull-back is the pull-back of the first Chern class, we want to understand 
when $[\mathbb{A}^{1}_{\chi} \times W/_{k}\mathcal{G}] \to [W/_{k}\mathcal{G}]$ is the trivial line bundle. Equivalently, we want to analyze when there is 
an equivariant isomorphism of line bundles $\Lambda:W\times \mathbb{A}^{1}_{\Id} \to W\times \mathbb{A}^1_{\chi}$.
$$
\xymatrix{
W\times\mathbb{A}^1_{\Id} \ar[rd] \ar[rr]^{\Lambda} & & W\times \mathbb{A}^1_{\chi} \ar[ld] \\
& W &
}
$$
To answer this question we use ideas from \cite{Bri}.
We have that $\Lambda(w, t) = (w, f(w, t))$ for some function $f$. Since we want $\Lambda$ to be a map of line bundles, for each $w\in W$, the map $f_{w}:\mathbb{A}^1\to\mathbb{A}^1$ given by $t\mapsto f(w, t)$ has to be linear. As a result, $f(w, t) = t f(w, 1)$. Since $\Lambda$ is an isomorphism, $f$ needs to be invertible on each fiber, namely $f(w, 1)\neq 0$ for each $w\in W$. Thus, the assignment $w\mapsto f(w, 1)$ gives rise to an invertible function on $W$, which we still call $f: W\to\mathbb{G}_m$. 

Asking for $\Lambda$ to be $\mathcal{G}$-equivariant is the same as requiring the following equalities:
\begin{align*}
\Lambda(g\cdot(w, t)) = g\cdot \Lambda(w, t) 
&\ \Leftrightarrow \ \Lambda(g w, t) = g\cdot(w, f(w, t)) = (g w, \chi(g) f(w, t)) \\ 
&\ \Leftrightarrow \ (g w, f(g w, t))= (g w, \chi(g) f(w, t))
\end{align*}
Thus, $f(g w, t) = \chi(g) f(w, t)$. In particular, when $t=1$, we get
$f(g w, 1) = \chi(g) f(w, 1)$. Treating $f(w, 1)$ as $f(w)$ (so viewing $f$ as an invertible function on $W$), we get $f(gw)=\chi(g) f(w)$. Thus,
$$
\chi(g) = \frac{f(gw)}{f(w)}
$$
which is independent of the choice of $w$. Conversely, given any invertible function $f$ such that for every $w_1,w_2 \in W$, $\displaystyle{\frac{f(gw_1)}{f(w_1)}=\frac{f(gw_2)}{f(w_2)}}$, the character $\displaystyle{g \mapsto \frac{f(gw)}{f(w)}}$ is in the kernel of $A^{1}_{\mathcal{G}}\to A^{1}_{\mathcal{G}}(W)$. This shows that 
$$
\ker(A^{1}_{\mathcal{G}} \to A^{1}([W/_{k} \mathcal{G}])) \leftrightarrow \{f: W\to \mathbb{G}_m \ | \ \frac{f(g w_1)}{f(w_1)} = \frac{f(gw_2)}{f(w_2)} \text{ for all } w_1, w_2\in W\}$$

\begin{Lemma}
 Let $f \in \mathcal{O}_W(W)^*$ such that for every $g \in \mathcal{G}$ and every $w_1, w_2 \in W$,
 $$\frac{f(gw_1)}{f(w_1)}=\frac{f(gw_2)}{f(w_2)}$$

 Then $\displaystyle{f=c\prod_{i<j}(x_{i}y_{j}-x_{j}y_{i})^r}$ for some $r \in \mathbb{Z}$ and $c \in \mathbb{G}_m$.
 \end{Lemma}
 
 \begin{proof}
  We will say that $h\in \mathcal{O}_W(W)^*$ satisfies ($\ast$) if for every $g \in \mathcal{G}$ and every $w_1, w_2 \in W$,
 $\displaystyle{\frac{h(g w_1)}{h(w_1)}=\frac{h(g w_2)}{h(w_2)}}$. 
The set of functions satisfying ($\ast$) is a multiplicative subgroup of $\cO_W(W)^*$. Moreover, notice that
 $f_1:=\prod_{i<j}(x_{i}y_{j}-x_{j}y_{i})$ does satisfy ($\ast$). 
 
 Suppose that $f' \in \cO_W(W)^*$ satisfies ($\ast$). Since $f'$ is invertible, there are $a_{i,j}$ and $c$ such that
 $f'=c\prod_{i<j}(x_{i}y_{j}-x_{j}y_{i})^{a_{i,j}}$.
 Let $i<j$ be such that $a_{i,j}$ is the smallest exponent.
 
 Assume, to the contrary, that there are $l<m$ such that $a_{i,j} < a_{l,m}$.
 We can assume, after reparametrization, that $i=0$ and $j=1$.
 Then also $f:=f'f_0^{-a_{0,1}}$ satisfies ($\ast$). But now $f$ is of the form $f=c \prod_{i<j}(x_{i}y_{j}-x_{j}y_{i})^{b_{i,j}}$, with
 $b_{0,1}=0$, $b_{i,j}\ge 0$ for every $i,j$, and $b_{l,m}>0$.
 
 Fix $w_2 \in W$ and $\sigma \in S_{n+1}$ that sends $0 \mapsto l$ and $1 \mapsto m$ and let $g=(\Id,(1,...,1), \sigma) \in \mathcal{G}$.
 Now $f$ and $g^{-1}*f$ lift to sections of $\mathcal{O}_{\mathbb{A}^{2n+2}}(\mathbb{A}^{2n+2})$ (which we will call again $f$ and
 $g^{-1}*f$).
 Since $f$ satisfies ($\ast$), for every $w_1 \in W$,
 $f(gw_1)f(w_2)=f(gw_2)f(w_1)$.  Since
 $\cO_{\mathbb{A}^{2n+2}}(\mathbb{A}^{2n+2}) \to \cO_{W}(W)$ is injective, also in $\cO_{\mathbb{A}^{2n+2}}(\mathbb{A}^{2n+2})$ we have
 $f(gw_1)f(w_2)=f(gw_2)f(w_1)$ for every $w_1 \in \mathbb{A}^{2n+2}$. But if we pick $p \in V(x_{0}y_{1}-y_{0}x_{1})\smallsetminus
 \bigcup_{i<j, (i,j) \neq (0,1)} V(x_{i}y_{j}-x_{j}y_{i})$, then $0=f(gp)f(w_2)$ and $f(gw_2)f(p) \neq 0$ as $gw_2 \in W$ and $f \in \cO_W(W)^*$. This is the contradiction.\end{proof}
 
 The lemma shows that the kernel of $A^{1}_{\mathcal{G}} \to A^{1}([W/_{k} \mathcal{G}])$ is generated by 
 $f=\prod_{i<j} (x_{i} y_{j} - x_{j}y_{i})$. Now we compute the character $g\mapsto \dfrac{f(gw)}{f(w)}$. From Lemma \ref{Lemma:A1G:and:A1Gk}
 this character can be identified with a triple $(\ell, a, \delta)\in \bZ\times\bZ\times(\bZ/2\bZ)$. To extract $\ell\in\bZ$, choose 
 $g=(A, (1, ..., 1), \operatorname{id})$. Note that
 $$
 (x_{i}y_{j} - x_{j} y_{i})
 \begin{pmatrix}
 a_{0} & \cdots & a_{n} \\
 b_{0} & \cdots & b_{n}
 \end{pmatrix} = \det \begin{pmatrix} 
 a_{i} & a_{j} \\
 b_{i} & b_{j}
 \end{pmatrix}
 $$
 Therefore, 
$$
(x_{i}y_{j} - x_{j} y_{i})
 \left( A \begin{pmatrix}
 a_{0} & \cdots & a_{n} \\
 b_{0} & \cdots & b_{n}
 \end{pmatrix}\right) = \det(A) \det\begin{pmatrix} 
 a_{i} & a_{j} \\
 b_{i} & b_{j}
 \end{pmatrix}
 $$
 As a result, $\dfrac{f(gw)}{f(w)} = \det(A)^{\binom{n+1}{2}} \ \Rightarrow \ \ell = {\binom{n+1}{2}}$. To extract $a$ and $\delta$, choose
 $g=(\operatorname{id}, (\lambda, ..., \lambda), \sigma)$; then, 
\begin{align*}
g\mapsto \frac{f(gw)}{f(w)} = 
\frac{\prod_{i<j} (\lambda^{2k} x_{\sigma^{-1}(i)} y_{\sigma^{-1}(j)} - \lambda^{2k}x_{\sigma^{-1}(j)} y_{\sigma^{-1}(i)})}{\prod_{i<j} (x_{i} y_{j} - x_{j}y_{i})}
&=\lambda^{2k \binom{n+1}{2}} \operatorname{sgn}(\sigma) \\ 
&= \lambda^{kn(n+1)}\operatorname{sgn}(\sigma) 
\end{align*}
It follows that $a=kn$ and $\delta=1$. This verifies the first statement of the proposition, and the second one follows immediately.\end{proof}

\begin{Teo}\label{Theo:Pic}
For every $k\neq 0$, we have $\Pic([W/G_{k}]) =\bZ/dn\bZ$ where $d=\gcd(2k, n+1)$. 
\end{Teo}

As a corollary, when $k=-2, -1$, after applying Theorem \ref{main:theorem}  and Proposition \ref{Proposition:iso:W/G:and:M0n}, we get the Picard groups
of $[\mathcal{U}/\PGL_{n+1}]$ and $[U^\circ/\PGL_2]$: 
$$
\Pic([\mathcal{U}/\PGL_{n+1}])\cong\Pic([W/G_{-2}]) \cong
\begin{cases}
\bZ/n\bZ &\text{if } n \text{ is even} \\
\bZ/2n\bZ &\text{if } n\equiv 1 \text{ (mod } 4) \\
\bZ/4n\bZ &\text{if } n\equiv 3 \text{ (mod } 4)
\end{cases} 
$$
and
$$
\Pic([U^{\circ}/\PGL_{2}]) \cong \Pic([W/G_{-1}]) \cong
\begin{cases}
\bZ/n\bZ &\text{if } n \text{ is even} \\
\bZ/2n\bZ &\text{if } n \text{ is odd} 
\end{cases}
$$
This proves Theorem \ref{Teo:answer:pic}.
\begin{proof}[Proof of Theorem \ref{Theo:Pic}:]
From Lemma \ref{Lemma:A1G:and:A1Gk}, $A^{1}_{G_{k}} \cong \{(\ell, a, \delta): 2k\ell = (n+1)a\} \subseteq \mathbb{Z} \oplus \mathbb{Z} \oplus (\mathbb{Z}/2\mathbb{Z})$, whereas
from Proposition \ref{Prop:Kernel},
$\operatorname{ker}(A^1_{G_{k}}\to A^1_{G_{k}}(W))=\left\langle \left( \binom{n+1}{2}, kn,1\right) \right \rangle$.

Consider the isomorphism 
$$
\bZ\oplus (\bZ/2\bZ) \overset{\Psi}{\to} \{(\ell, a, \delta)\in\bZ\oplus\bZ\oplus(\bZ/2\bZ) : 2k\ell = (n+1)a\}
$$
that sends
$$
(a, \delta) \mapsto \left(\frac{n+1}{d} a, \frac{2k}{d} a, \delta\right)
$$
where $d=\gcd(2k, n+1)$. Then $\Psi^{-1}(\binom{n+1}{2}, kn, 1) = \left(\frac{nd}{2}, 1\right)$. Therefore, 
$$
A^{*}_{G_{k}}(W) = \frac{\bZ\oplus\bZ/2\bZ}{(nd/2, 1)} \cong \frac{\bZ}{(nd)\bZ}
$$
where the last isomorphism is induced by $\mathbb{Z} \oplus (\mathbb{Z}/2\mathbb{Z}) \to \mathbb{Z}/2m\mathbb{Z}$, $(a, \delta) \mapsto a+m\delta$ with $m=\dfrac{nd}{2}$.
\end{proof}

\subsection{Maps between the Picard groups.} Now that we have determined the Picard groups $\Pic([W/G_{k}])$, it is worth 
analyzing the maps between them induced by $F_{b, c}: [W/G_{b}]\to [W/G_{c}]$ for $c\mid b$. The result is summarized in the following:

\begin{Prop} \label{prop:pic:map} The natural map $\Pic([W/G_{c}])\to \Pic([W/G_{b}])$, after identifying the Picard groups as in 
Theorem \ref{Theo:Pic}, is given by the multiplication
\begin{align*}
\bZ/(d_{c} n \bZ) &\to \bZ/(d_{b} n\bZ)  \\
1 &\mapsto \frac{d_{b}}{d_{c}} 
\end{align*}
where $d_{c}=\gcd(2c, n+1)$ and $d_{b}=\gcd(2b, n+1)$. In particular, it is injective.
\end{Prop}

\begin{proof}
Consider the homomorphism $\mathcal{G}\overset{\mathfrak{h}_{b, c}}{\longrightarrow} \mathcal{G}$ given by 
$(A, (\lambda_0, ..., \lambda_n), \sigma) \mapsto (A, (\lambda_{0}^{b/c}, ..., \lambda_{n}^{b/c}), \sigma)$. This induces
$h_{b, c}: G_{b}\to G_{c}$ and the following diagrams commute
$$
\begin{tikzcd} 
\mathcal{G} \ar[d] \ar{r}{\mathfrak{h}_{b, c}} & \mathcal{G} \ar[d] \\
G_{b} \ar{r}{h_{b, c}} & G_{c}
\end{tikzcd}
\ \ \ \rightsquigarrow \ \ \
\begin{tikzcd}
B\mathcal{G} \ar[r] \ar[d] & B\mathcal{G} \ar[d] \\
B G_{b} \ar[r] & B G_{c}
\end{tikzcd}
\ \ \ \rightsquigarrow \ \ \
\begin{tikzcd}
A^{1}(B\mathcal{G}) & A^{1}(B\mathcal{G}) \ar[l] \\
 A^{1}(B G_{b}) \ar[hook]{u} & A^{1}(B G_{c}) \ar[l] \ar[hook]{u}
\end{tikzcd}
$$
After identifying $A^1(B\mathcal{G})$, $A^1(BG_{b})$ and $A^1(BG_{c})$ with the respective character groups, the diagram above becomes:
$$
\renewcommand{\labelstyle}{\textstyle}
\xymatrix@R=5pc@C=5pc{
\bZ\oplus\bZ\oplus(\bZ/2\bZ) & &\bZ\oplus\bZ\oplus(\bZ/2\bZ) \ar[ll]_{(\ell,\frac{b}{c}a,\delta) \mapsfrom (\ell,a,\delta)} \\
\bZ\oplus(\bZ/2\bZ) \ar[u]^{(1,\delta) \mapsto (\frac{n+1}{d_b},\frac{2b}{d_b},\delta)} & & \bZ\oplus(\bZ/2\bZ) \ar[ll]^{(\frac{d_b}{d_c},\delta) \mapsfrom (1,\delta)} \ar[u]_{(1,\delta) \mapsto (\frac{n+1}{d_c},\frac{2c}{d_c},\delta)}
}
$$
The arrows above can be motivated as follows. Since the vertical arrows are injective, to understand the map on the bottom row it is enough to understand the one on the top row. This map sends the representation $\chi:\mathcal{G} \to \mathbb{G}_m$ to
$\chi \circ \mathfrak{h}_{b,c}:\mathcal{G} \to \mathbb{G}_m$. Then identifying $A^1(B\mathcal{G}) \cong \mathbb{Z} \oplus \mathbb{Z} \oplus (\mathbb{Z}/2\mathbb{Z})$ as in Lemma \ref{Lemma:A1G:and:A1Gk}, it sends $(\ell,a,\sigma) \mapsto (\ell,\frac{b}{c}a,\sigma)$. 

Consider then
$$
\begin{tikzcd}
BG_{b} \ar[r] & BG_{c} \\
{[W/G_{b}]} \ar[r] \ar[u] & {[W/G_{c}]} \ar[u]
\end{tikzcd}
\ \ \ \rightsquigarrow \ \ \ 
\begin{tikzcd}
A^{1}(BG_{b}) \ar[d] & A^{1}(BG_{c}) \ar[d] \ar[l] \\
A^{1}([W/G_{b}])  & A^{1}([W/G_{c}]) \ar[l]
\end{tikzcd}
$$
where the vertical arrows in the second diagram are surjective. After identifying these groups as in Lemma \ref{Lemma:A1G:and:A1Gk}, and using Proposition \ref{Prop:Kernel}, the conclusion follows from a diagram chase:
$$
\renewcommand{\labelstyle}{\textstyle}
\xymatrix@R=5pc@C=5pc{
\bZ\oplus(\bZ/2\bZ) \ar[d]_{(1,0) \mapsto [(1,0)]} & & \bZ\oplus(\bZ/2\bZ) \ar[ll]_{(\frac{d_b}{d_c},0) \mapsfrom (1,0)} \ar[d]^{(1,0) \mapsto [(1,0)]} \\
\dfrac{\bZ\oplus(\bZ/2\bZ)}{(d_{b}n/2, 1)} \ar[d]_{[(1,0)] \mapsto [1]} & &\dfrac{\bZ\oplus(\bZ/2\bZ)}{(d_{c}n/2, 1)} \ar[ll] \ar[d]^{[(1,0)] \mapsto [1]} \\
\bZ/d_{b}n\bZ & & \bZ/d_{c}n\bZ \ar[ll]^{[\frac{d_b}{d_c}] \mapsfrom [1]}
}
$$
\end{proof}

\section{Connections with the  hyperelliptic curves}\label{section:hyperelliptic}

Let $\mathcal{H}_{g}$ be the moduli stack of smooth hyperelliptic curves of genus $g$. In this section we construct a map $[\mathcal{U}/\PGL_{n+1}]\to\mathcal{H}_{g}$ where $n=2g+1$ which fits into a commutative triangle (recall from the conventions that, when we say that a diagram commutes, we mean it 2-commutes):
$$
\xymatrix{
[\mathcal{U}/\PGL_{n+1}] \ar[rr]^{\Phi} \ar[rd] & & [U^\circ/\PGL_{2}] \\
& \mathcal{H}_{g} \ar[ru] &
}
$$
Moreover, we show that the induced map $\Pic(\mathcal{H}_g)\to \Pic([\mathcal{U}/\PGL_{n+1}])$ is an isomorphism. 
The same triangle above has been studied at the level of coarse moduli spaces by Avritzer-Lange \cite{AvLa}. There are further connections between complete intersections and hyperelliptic curves that are not explored in this paper. Indeed,  given a smooth complete intersection $X=Q_1\cap Q_2$ in $\bP^{2g+1}$, consider the hyperelliptic curve $C$ given by the equation $y^2=\det(xQ_1+zQ_2)$ in $\bP(1, 1, g+1)$. In his thesis \cite{Reid} Reid proves that the Jacobian $J(C)$ is isomorphic to the Fano variety $S$ of $(g-1)$-planes contained in $X$. Furthermore Donagi \cite{Don} gives a geometric construction for the corresponding group law on $S$. 

We proceed by constructing a map of stacks $[\mathcal{U}/\PGL_{n+1}] \to \mathcal{H}_g$.  Recall that the points of $\mathcal{F}$ can be identified with pairs $(Q_{1}, Q_{2})$ where $Q_{1}$ and $Q_{2}$ form a basis for a $2$-dimensional subspace $V\in\mathcal{U}\subseteq\operatorname{Gr}(2, N)$. 
To a pair $(Q_{1}, Q_{2})$ we can associate a hyperelliptic curve $y^2=\det(x Q_{1} + z Q_{2})$ of genus $g$ which naturally lives in a weighted projective space $\mathbb{P}(1, 1, g+1)$ with coordinates $x, z$ and $y$ respectively. 

Let us define the universal hyperelliptic curve $\mathcal{C}\subseteq \mathcal{F}\times\mathbb{P}(1, 1, g+1)$ as follows:
$$
\mathcal{C} = \{ \left( (Q_{1}, Q_{2}), (x, z, y) \right) \in \mathcal{F}\times\mathbb{P}(1, 1, g+1): 
y^2=\det(x Q_{1} + z Q_{2}) \}
$$
The projection map $\pi_1: \mathcal{C} \to \mathcal{F}$ is flat. This follows at once from the miracle flatness theorem (\cite{Hart}*{Exercise III.10.9}) as the fibers of $\pi_1$ have dimension 1, $\mathcal{C}$ is Cohen-Macaulay (as it is a hypersurface) and $\mathcal{F}$ is smooth (being a $\GL_{2}$-bundle over $\cU$).   

Furthermore, $\pi: \mathcal{C}\to \mathcal{F}$ is smooth because it has smooth fibers. This can be checked using the Jacobian criterion. The conclusion is that $\mathcal{C}\to\mathcal{F}$ is a family of hyperelliptic curves of genus $g$.

Next, we will find a group $\Gamma$ which acts on both $\mathcal{C}$ and $\mathcal{F}$ and makes the map $\pi: \mathcal{C}\to \mathcal{F}$ equivariant. This group will have the additional property that $[\mathcal{F}/\Gamma] = [\mathcal{U}/\PGL_{n+1}]$. Define
$$
\Gamma := \frac{\GL_{2}\times\GL_{n+1}}{\left\langle( \diag(\lambda^2,\lambda^2), \diag(\lambda,...,\lambda)\right)\rangle}
$$
Define the action of $\Gamma$ on $\mathcal{F}$ as follows. Given an element $(Q_{1}, Q_{2})$ of $\mathcal{F}$ and $(M, A)\in \Gamma$ where $M=\begin{pmatrix} a & b \\ c & d \end{pmatrix}$, we define
$$
(M, A)\ast (Q_{1}, Q_{2}) := ( (A^{-1})^{T}(a Q_{1} + b Q_{2})A^{-1}, 
(A^{-1})^{T}(c Q_{1} + d Q_{2}) A^{-1})
$$
and $\Gamma$ acts on a point $((Q_{1}, Q_{2}), (x, z, y))\in \mathcal{C}$ as follows:
$$
(M, A) \ast ((Q_{1}, Q_{2}), (x, z, y)) = \left((M, A)\ast (Q_{1}, Q_{2}), 
(M^{-1})^{T} \begin{pmatrix} x \\ z \end{pmatrix}, \frac{y}{\det A}\right)
$$
Let us observe that the latter point lies on $\mathcal{C}$. Denote $y'=\dfrac{y}{\det(A)}$ and 
$\begin{pmatrix} x' \\ z' \end{pmatrix} = (M^{-1})^{T}\begin{pmatrix} x \\ z \end{pmatrix}$. We need to verify that 
\begin{align*}
(y')^2 &= \det(x'(A^{-1})^T(aQ_1+bQ_2)A^{-1}+z'(A^{-1})^T(cQ_1+dQ_2)A^{-1})
\end{align*}
Indeed, 
\begin{align*}
\text{RHS}&=\det(A)^{-2}\det(x'(aQ_1+bQ_2)+z'(cQ_1+dQ_2))\\
&=\det(A)^{-2}\det\left(
\begin{pmatrix} x' & z' \end{pmatrix}
\begin{pmatrix} a & b \\ c & d \end{pmatrix} 
\begin{pmatrix} Q_1 \\ Q_2 \end{pmatrix} \right) \\
&=\det(A)^{-2}\det\left(
\left((M^{-1})^{T} \begin{pmatrix} x \\ z \end{pmatrix}\right)^{T}
M
\begin{pmatrix} Q_1 \\ Q_2 \end{pmatrix} \right) \\ 
&=\det(A)^{-2}\det\left(\begin{pmatrix} x & z \end{pmatrix} \begin{pmatrix} 
1 & 0 \\
0 & 1
\end{pmatrix} \begin{pmatrix} Q_1 \\ Q_2 \end{pmatrix} \right) \\
&=\det(A)^{-2}\det(x Q_{1} + zQ_{2}) 
=\frac{y^2}{\det(A)^2} = (y')^2
\end{align*}
as desired.

It is clear that $\pi_{1}: \mathcal{C}\to \mathcal{F}$ is $\Gamma$-equivariant. Consider the subgroup of $\Gamma$ generated by the pairs
$$
(\diag(\lambda, \lambda), \diag(1, 1, ..., 1))
$$
with $\lambda\neq 0$. This subgroup is isomorphic to $\mathbb{G}_m$. Note that
$$
\Gamma/\mathbb{G}_m \cong \frac{\GL_{2}\times\GL_{n+1}}{\langle (\diag(\lambda, \lambda), \diag(\mu,...,\mu)) \rangle} \cong \PGL_{2}\times\PGL_{n+1}
$$
It is worth stressing that the action of $\mathbb{G}_{m}$ on $\mathcal{F}$ and the induced action of $\PGL_2 \times \PGL_{n+1}$ on $\mathcal{F}^{\circ}$ are the same as the one introduced in Section \ref{section:main:diag}.
As a result, 
\begin{align*}
[\mathcal{F}/\Gamma] \cong 
[[\mathcal{F}/\mathbb{G}_m]/(\Gamma/\mathbb{G}_m)] &\cong 
[[\mathcal{F}/\mathbb{G}_m]/\PGL_{2}\times\PGL_{n+1}] \\
&\cong[\mathcal{F}^{\circ}/\PGL_{2}\times\PGL_{n+1}] \cong [\mathcal{U}/\PGL_{n+1}]
\end{align*}

Thus, $\mathcal{C}\to\mathcal{F}$ induces a map $[\mathcal{C}/\Gamma] \to [\mathcal{F}/\Gamma] \cong [\mathcal{U}/\PGL_{n+1}]$. This realizes $[\mathcal{C}/\Gamma]$ as a family of hyperelliptic curves over the base $[\mathcal{U}/\PGL_{n+1}]$. By definition, we get a morphism $[\mathcal{U}/\PGL_{n+1}]\to \mathcal{H}_{g}$.

\begin{Rmk}The morphism $[\mathcal{U}/\PGL_{n+1}] \to \mathcal{H}_g$ can be also understood as follows. Given a point $\spec(k) \to [\mathcal{U}/\PGL_{n+1}]$, we can lift it to $\spec(k) \to \mathcal{U}$. The latter corresponds to a pencil of quadrics $\mathcal{Q} \to \mathbb{P}^1$. Consider the isotropic Grassmannian $\operatorname{Gr}(\mathcal{Q}):=\{(L,Q):Q$ is a fiber of $\mathcal{Q} \to \mathbb{P}^1$, and $L \subseteq Q$ is a linear subspace of dimension $\frac{n-1}{2} \}$. This comes with a map $\operatorname{Gr}(\mathcal{Q}) \to \mathbb{P}^1$. Its Stein factorization is a hyperelliptic curve $\mathcal{E} \to \mathbb{P}^1$, ramified at the points of $\mathbb{P}^1$ which correspond to singular members of the pencil $\mathcal{Q} \to \mathbb{P}^1$ (\cite{Reid}*{Theorem 10.1, b)}). The point $\spec(k) \to [\mathcal{U}/\PGL_{n+1}]$ gets sent through $[\mathcal{U}/\PGL_{n+1}] \to \mathcal{H}_g$ to the isomorphism class of $\mathcal{E}$.
\end{Rmk}

Next, we show that the following diagram commutes:
$$
\xymatrix{
[\mathcal{U}/\PGL_{n+1}] \ar[rr] \ar[rd] & & [U^\circ/\PGL_{2}] \\
& \mathcal{H}_{g} \ar[ru] &
}
$$
After identifying $[\mathcal{U}/\PGL_{n+1}]\cong [\mathcal{F}/\Gamma]$, consider the diagram:
$$
\xymatrix{
 \mathcal{F} \ar[rr] \ar[d] & & U^\circ \ar[d] \\
[\mathcal{F}/\Gamma] \ar[rr] \ar[rd] & & [U^\circ/\PGL_{2}] \\
 & \mathcal{H}_{g} \ar[ru] &
}
$$

Observe that $\mathcal{F}\to [\mathcal{F}/\Gamma]$ is an atlas. \begin{Lemma} \label{lemma:pentagon} The big pentagon above commutes.\end{Lemma}
\begin{proof}
Indeed, $\mathbb{P}(H^0(\mathcal{O}_{\bP^{1}}(n+1)))$ can be understood as the Hilbert scheme of $n+1$ points (counted with multiplicity) in $\mathbb{P}^1$. Then a morphism $B\to U^{\circ}$ corresponds to a closed subscheme $\mathcal{D}\subseteq B\times\bP^{1}$ flat over $B$ such that for every $b\in B$, $\mathcal{D}_{b}\subseteq\bP^{1}_{b}$ consists of $n+1$ distinct points. Similarly, a morphism $B\to [U^{\circ}/\PGL_{2}]$ corresponds to a pair $(\mathcal{E}, \mathcal{D})$ where $\mathcal{E}\to B$ is a proper flat family of smooth genus $0$ curves, $\mathcal{D}\subseteq\mathcal{E}$ is an effective Cartier divisor, $\mathcal{D}\to B$ is \'{e}tale, of degree $n+1$ (See \cite{GorViv}*{Definition 3.3}).

The morphism $\mathcal{F}\to U^\circ$ then corresponds to the family $\{\det(x_0 Q_1+x_1 Q_2)=0\} \subseteq \mathcal{F}\times\bP^{1}$
of divisors on a $\bP^1$-bundle over $\mathcal{F}$. 

The map $\mathcal{H}_{g} \to [U^\circ/\PGL_{2}]$ sends a family of hyperelliptic curves $\mathcal{X}\to B$ to $(\mathcal{X}/\iota, \mathcal{D}_{\mathcal{X}})\to B$ where $\iota$ is the hyperelliptic involution and $\mathcal{D}_{\mathcal{X}}$ is the fixed scheme of $\iota$. Observe that $\mathcal{D}_{\mathcal{X}}$ is supported on the image of the ramification locus of the quotient map $\mathcal{X} \to \mathcal{X}/\iota$, and the induced morphism $\mathcal{D}_{\mathcal{X}}\to B$ is \'{e}tale of  degree $2g+2$. Our goal is to show that $(\mathcal{C}/\iota, \mathcal{D}_{\mathcal{C}})\to \mathcal{F}$ is isomorphic to the previous family, namely
$$
(\mathcal{F}\times\bP^1, \{\det(x_0 Q_1+x_1 Q_2)=0 \})
$$
Recall that $\mathcal{C}$ is explicitly described as a zero locus of the universal hyperelliptic curve $y^2 = \det(xQ_1 + zQ_2)$. In this case, $\iota: \mathcal{C}\to \mathcal{C}$, $((Q_1, Q_2), (x, z, y))\mapsto ((Q_1, Q_2), (x, z, -y))$ is the hyperelliptic involution. The map $\mathcal{C}\to \mathcal{F}\times\bP^{1}$, 
$((Q_1, Q_2), (x, z, y))\mapsto ((Q_1, Q_2), (x, z))$ is the quotient by $\iota$. The subschemes $\mathcal{D}_{\mathcal{C}}$ and $\{\det(x_0 Q_1+x_1 Q_2)=0\}\subset\mathcal{F}\times\bP^1$ are the same because they agree over each fiber of $(Q_1, Q_2)\in\mathcal{F}$. This proves the desired commutativity. \end{proof}
Now, let $g:[\cU/\PGL_{n+1}] \to \mathcal{H}_g \to [U^\circ /\PGL_2]$ be the composition. Our goal is to show that $g$ and $\Phi$ are isomorphic.

Recall, from Proposition \ref{Proposition:iso:W/G:and:M0n}, that $[U^\circ/\PGL_2] \cong [M_{0,n+1}/S_{n+1}]$.
In particular, there is a dense open embedding $V \to [U^\circ/\PGL_2]$ with $V$ a scheme. 
Since $\Phi$ and $g$ agree set theoretically, the two pull-backs $[\cU/\PGL_{n+1}] \times_{\Phi,[U^\circ /\PGL_2]} V$ and $[\cU/\PGL_{n+1}] \times_{g,[U^\circ /\PGL_2]} V$ are the same open substacks of $[\cU/\PGL_{n+1}]$. We will denote this pull-back by $\mathcal{X}$.
$$
\xymatrix{ 
\mathcal{F} \ar[r] & [\mathcal{U}/\PGL_{n+1}]  \ar@<-.5ex>[r]_{g} \ar@<.5ex>[r]^{\Phi} & [U^{\circ}/\PGL_{2}] \\ 
\mathcal{F}|_{\mathcal{X}} \ar[u] \ar[r] & \mathcal{X} \ar[u]  \ar@<-.5ex>[r]_{g|_{\mathcal{X}}} \ar@<.5ex>[r]^{\Phi|_{\mathcal{X}}} & V \ar[u]
}
$$
Since $\mathcal{F} \to [\cU/\PGL_{n+1}] \xrightarrow{\Phi} [U^\circ /\PGL_2]$ and $\mathcal{F} \to [\cU/\PGL_{n+1}] \xrightarrow{g} [U^\circ /\PGL_2]$ are isomorphic arrows, also
$\mathcal{F}_{|\mathcal{X}} \to \mathcal{X} \xrightarrow{\Phi_{|\mathcal{X}}} V$ and $\mathcal{F}_{|\mathcal{X}} \to \mathcal{X} \xrightarrow{g_{|\mathcal{X}}} V$ are isomorphic. In particular, since $V$ is a scheme, they are the same arrow. But then $\Phi_{|\mathcal{X}}=g_{|\mathcal{X}}$ since they agree once pulled back to an atlas, and $V$ is a sheaf in the smooth topology.
Then $g$ and $\Phi$ are isomorphic from \cite{lemma:extension}*{Lemma 7.2}.

\begin{Prop} \label{prop:pic:isom}
Consider the map $[\mathcal{U}/\PGL_{n+1}] \cong [\mathcal{F}/\Gamma]\to \mathcal{H}_{g}$ constructed above. The induced map $\Pic(\mathcal{H}_{g})\to \Pic([\mathcal{U}/\PGL_{n+1}])$ is an isomorphism. 
\end{Prop}

\begin{proof}
From the commutative triangle, we have the induced map on Picard groups:
$$
\xymatrix{
\Pic([\mathcal{U}/\PGL_{n+1}]) & & \Pic([U^{\circ}/\PGL_{2}]) \ar[ll] \ar[ld] \\
& \Pic(\mathcal{H}_g) \ar[lu] & 
}
$$
When $g$ is even, then $n=2g+1 \equiv 1 \text{ (mod } 4)$. In this case, $\Pic([U^{\circ}/\PGL_{2}])\to \Pic(\mathcal{H}_g)$ is an isomorphism \cite{GorViv}*{Theorem 3.6}, and $\Pic([U^{\circ}/\PGL_{2}])\to \Pic([\mathcal{U}/\PGL_{n+1}])$ is an isomorphism (Proposition \ref{prop:pic:map}). So the resulting map $\Pic(\mathcal{H}_{g}) \to \Pic([\mathcal{U}/\PGL_{n+1}])$ is an isomorphism. 

If $g$ is odd, then $n=2g+1 \equiv 3 \text{ (mod } 4)$. In this case, $\alpha:\Pic([U^{\circ}/\PGL_{2}])\to \Pic(\mathcal{H}_g)$ is multiplication by $2$ \cite{GorViv}*{Theorem 3.6}, and so is $\beta:\Pic([U^{\circ}/\PGL_{2}])\to \Pic([\mathcal{U}/\PGL_{n+1}])$ from Proposition \ref{prop:pic:map}. The above commutative triangle becomes:
$$
\begin{tikzcd} 
\bZ/4n\bZ \ar[hookleftarrow]{rr}{\beta} & & \bZ/2n\bZ  \\
& \bZ/4n\bZ \ar[hookleftarrow]{ru}[swap]{\alpha} \ar{lu}{h} & 
\end{tikzcd}
$$
We want to show that $h$ is bijective. It is enough to show that $h$ is surjective. Assume, to the contrary, that $\text{im}(h)\subsetneq \bZ/4n\bZ$. Since $|\text{im}(h)| \mid 4n$ and $2n\leq \text{im}(h)<4n$, we conclude that $|\text{im}(h)|=2n$. But then $\ker(h)=\langle 2n\rangle$ is the unique subgroup in $\bZ/4n\bZ$ of size $2$. We have
$$
0=h(2n)=h(\alpha(n))=\beta(n) \neq 0
$$
which is a contradiction.
\end{proof}

\appendix
\section{}
In this appendix we will prove Theorem \ref{Proposition:iso:normal:alg:stacks} which played a key role in the proof of Theorem \ref{main:theorem}. All the necessary technical tools will be recalled along the way. Our definition of Artin stacks differs from the one in Olsson's book \cite{Ols}
since we have two extra conditions. Namely, we require the diagonal to be of finite type and separated.

Lemma \ref{Lemma:Atlas} follows from the fact that inertia commutes with base change. For completeness we report a proof below. See \cite{Ols}*{Section 3.4.9.} for the relevant definitions.
\begin{Lemma}\label{Lemma:Atlas}
 Let $f:\mathcal{X}_1 \to \mathcal{X}_2$ be a morphism of Artin stacks, and let $p \in |\mathcal{X}_1(\operatorname{Spec}(k))|$.
 Let $i:U \to \mathcal{X}_2$ be a morphism where $U$ is an algebraic space and assume that there is a point $q\in U(\spec(k))$ such that
 $i(q)$ is isomorphic to $f(p)$ through $\sigma$. Let $F:=\mathcal{X}_{1}\times_{\mathcal{X}_2} U$ with the two projections $\pi_1$ and $\pi_2$,
 and let $z:=(p,q,\sigma) \in F$, where $\sigma: f(p)\to i(q)$.
  \vspace{-0.2em} 
 $$
 \!
 \xymatrix{
 F \ar[r] \ar[d]_{\pi_2} \ar[r]^{\pi_1} & \mathcal{X}_1 \ar[d]^{f} \\
 U \ar[r]_-{i} & \mathcal{X}_2
 }
 $$
 Then the following sequence is exact: 
  \vspace{-0.6em}
 $$1 \longrightarrow \operatorname{Aut}(z) \overset{\pi_1}{\longrightarrow} \operatorname{Aut}(p) \overset{f}{\longrightarrow} \operatorname{Aut}(f(p))$$ 
\end{Lemma}
\begin{proof}
 
As $U$ is an algebraic space, $U\to \mathcal{X}_2$ is representable. Being representable is stable under base change so
 $\pi_1$ is representable. The injectivity of the map $\operatorname{Aut}(z)\to\operatorname{Aut}(p)$ follows from \cite{AH}*{Lemma 2.3.9}.
 
 $\operatorname{Ker}(f) \subseteq \operatorname{Im}(\pi_1)$:
 Assume that $\tau \in \operatorname{Aut}(p)$ goes to the identity through $f$. Then the element $(\tau,\Id_q)$ is an automorphism of $z$ mapping to $\tau$.
 
 $\operatorname{Im}(\pi_1) \subseteq \operatorname{Ker}(f)$: Let
 $\mu \in \operatorname{Aut}(z)$. Then $\mu=(\tau_1,\tau_2) \in \operatorname{Aut}(p) \times \operatorname{Aut}(q)$,
 since $U$ is an algebraic space $\tau_2=\Id_q$. Moreover $\mu$ makes this diagram commutative:
 \begin{center}
$\xymatrix{\ar @{} [dr]
f(p) \ar[d]_{f(\tau_1)} \ar[r]^{\sigma}  & i(q)\ar[d]^{i(\Id_q)} \\
f(p) \ar[r]^{\sigma} & i(q) 
}
$ 
\end{center}
This means that $f(\pi_2(\mu))=f(\tau_1)=\Id$, namely $\pi_2(\mu) \in \operatorname{Ker}(f)$.
\end{proof}

It is well known that a bijective separated morphism between two normal equidimensional schemes
of finite type over any algebraically closed field of characteristic 0 is an isomorphism.
We want to generalize this statement to stacks. In doing so, we need to control a morphism from the induced function between the $k$-points.
The following three lemmas go in this direction; these are known to the experts, but we have not been able to find a suitable reference.

It is well-known that a morphism (between schemes of finite type over a field) which is surjective on closed points, is surjective. The following lemma generalizes this statement.

\begin{Lemma}\label{lemma:surj:closed:pts}
  Let $k$ be an algebraically closed field and let
  $f:X_1 \to X_2$ be a morphism of algebraic spaces of finite type over $\spec(k)$. Assume that for every morphism $\spec(k) \to X_2$, the fiber product
  $X_1\times_{X_2}\spec(k)$ is not empty (i.e. assume $f$ surjective on $k$-points).
  
  Then $f$ is surjective.
  
 \end{Lemma}
\begin{proof}Let $j_2: V_2 \to X_2$ be an atlas, let $V_1$ be an atlas of $V_2 \times_{X_2}X_1$, and $j_1: V_1\to X_1$ be the composition.
$$
\xymatrix{
V_{1} \ar[r] \ar[rd]_{j_1} \ar@/^1.5pc/[rr]^{g} & V_{2}\times_{X_2} X_1 \ar[d] \ar[r] & V_2 \ar[d]^{j_2} \\
& X_1 \ar[r]_{f} & X_2
}
$$
Let $g:V_1 \to V_2$ the corresponding
map between atlases. In particular, we have $f \circ j_1=j_2 \circ g$.

By construction, $g$ is surjective on $k$-points, and thus $g$ is surjective. Now the surjectivity of $f$ follows
 from $f \circ j_1=j_2 \circ g$.
\end{proof}
 \begin{Lemma}\label{inj:closed:pts}
  Let $k$ be an algebraically closed field
  and let $f:X_1 \to X_2$ be a morphism of algebraic spaces of finite type over $\spec(k)$. Assume that $X_1(k) \to X_2(k)$ is injective.
  
  Then for every field $L$, $X_1(L) \to X_2(L)$ is injective.
 \end{Lemma}
\begin{proof}From \cite{stacks-project}*{Lemma 58.19.2}, it is enough to show that $X_1 \to X_1 \times_{X_2}X_1$ is surjective.
By hypothesis it is surjective on $k$-points, from Lemma \ref{lemma:surj:closed:pts} it is surjective.
\end{proof}
 
\begin{Lemma}\label{Enough_closed_pts}
 Let $\mathcal{X}$ be an Artin stack of finite type over an algebraically closed field $k$ of characteristic 0. Assume
 that any object $p$ of $\mathcal{X}(\operatorname{Spec}(k))$ is such that $\operatorname{Aut}(p)=\{\Id\}$.
 
 Then $\mathcal{X}$ is an algebraic space.
\end{Lemma}

\begin{proof}
If we can show that $\Aut(p)=\{\Id\}$ for every $p\in\mathcal{X}(\operatorname{Spec}(L))$ for every algebraically closed field $L$, then \cite{B.Con}*{Theorem 2.2.5} applies and shows that $\mathcal{X}$ is an algebraic space. 

Let $i:U\to \mathcal{X}$ be an atlas of $\mathcal{X}$ and let $I:=\mathcal{X} \times_{\mathcal{X} \times \mathcal{X}}U$, where
 the map $\mathcal{X} \to \mathcal{X} \times \mathcal{X}$ is the diagonal, whereas the morphism $U  \to \mathcal{X} \times \mathcal{X}$
 is $(i,i)$. $I$ is an algebraic space of finite type over $\operatorname{Spec}(k)$.
 
 \begin{center}
$\xymatrix{\ar @{} [dr]
I\ar[d] \ar[r]  & U \ar[d]^{(i,i)} \\
\mathcal{X}\ar[r]_-{\text{diag}} & \mathcal{X} \times \mathcal{X} }$\end{center}
 
 We first observe that $I \to U$ is injective. Given a point $p:\operatorname{Spec}(k) \to U$, the fiber $I \times_U \spec(k) \cong \operatorname{Aut}(i\circ p)=\{\Id\}$ is just a point. Thus, $I \to U$ is a morphism between two algebraic spaces of finite type over $\spec(k)$, which is injective when restricted to the
 $k$-points. Then from Lemma \ref{inj:closed:pts} it is injective on $L$-points for every field $L$.
 
 Let then $L$ be an algebraically closed field, with a morphism $q:\spec(L)\to
 U$. By definition of fibred product and from the injectivity of $I(L)\to U(L)$,
 $\operatorname{Aut}(i\circ q)=\spec(L) \times_U I$ has a single point $L$-point,
 $z:\spec(L) \to \spec(L) \times_U I$.
 Now Lemma \ref{lemma:surj:closed:pts} applies, so $z$ is surjective.
 Then $\spec(L) \times_U I \to \spec(L)$ is quasifinite since it is quasifinite once precomposed with $z$.
Since $\mathcal{X} \xrightarrow{\text{diag }} \mathcal{X} \times \mathcal{X}$ is separated, $\spec(L) \times_U I \to \spec(L)$ is also separated. Thus from \cite{Ols}*{Theorem 7.2.10}, $\spec(L) \times_U I $ is a scheme.
 It is also a group since it is $\operatorname{Aut}(i\circ q)$, hence smooth since the characteristic is 0. In particular it is reduced with a single point, i.e. it is $\{\Id\}$.
 
 We just showed that for every algebraically closed field $L$, any object of $\mathcal{X}(\spec(L))$ has no nontrivial automorphism. \end{proof}

\begin{Teo}\label{Proposition:iso:normal:alg:stacks}Let $k$ be an algebraically closed field of characteristic 0.
 Let $\mathcal{X}_1$, $\mathcal{X}_2$ two normal, equidimensional, Artin stacks of finite type over $k$.
 Let $f:\mathcal{X}_1 \to \mathcal{X}_2$ be a separated morphism
 such that $\mathcal{X}_1(\operatorname{Spec}(k)) \to \mathcal{X}_2(\operatorname{Spec}(k))$
 is an equivalence (i.e. fully faithful and essentially surjective).
 
 Then $f$ is an isomorphism.
\end{Teo}

\begin{Rmk} One can understand Theorem \ref{Proposition:iso:normal:alg:stacks} as a consequence of Zariski's main theorem for representable morphisms of stacks \cite{LMB}*{Theorem 16.5}. The version above is advantageous because we do not need to check that $f$ is finite. Instead we assume that $\mathcal{X}_1$ and $\mathcal{X}_2$ are equidimensional. 
\end{Rmk}

\begin{proof}
Let $j:V \to \mathcal{X}_2$ be a smooth atlas where $V$ is a scheme, and
let $V':=V \times_{\mathcal{X}_2} \mathcal{X}_1$ be the fiber product, with the first projection
$\phi:V' \to V$. We want to show that $\phi$ is an isomorphism.

Since $\mathcal{X}_2$ is normal, $V$ is normal. Up to replacing a lower dimensional component $Z$ of $V$ with $Z \times \mathbb{P}^r$ for some $r$, we can assume that $V$ is equidimensional. Since also $\mathcal{X}_2$ is equidimensional and $j$ is smooth, the same conclusion holds for the fibers of $V \to \mathcal{X}_2$, thus also for the fibers of $V' \to \mathcal{X}_1$. But $\mathcal{X}_1$ is equidimensional, $V' \to \mathcal{X}_1$ is smooth, so also $V'$ is equidimensional.

\underline{$V'$ is an algebraic space}: From Lemma \ref{Lemma:Atlas}, for every
$q \in V'(\operatorname{Spec}(k))$ we have that $\operatorname{Aut}(q)=\Id$. From Lemma \ref{Enough_closed_pts}, $V'$ is an algebraic space. 

Since $\mathcal{X}_i$ are normal, $V'$ and $V$ are normal. Since $f$ is separated, $\phi$ is separated. We show that $\phi$ is bijective.

\underline{Surjective:} From Lemma \ref{lemma:surj:closed:pts} it is enough to show surjectivity on $k$-points,
i.e. it is enough to show that for every morphism
$\spec(k) \to V$, $V' \times_V \spec(k)$ is not empty. But $V'=\mathcal{X}_1 \times_{\mathcal{X}_2} V$ so it is enough that
 $\mathcal{X}_1 \times_{\mathcal{X}_2} \spec(k)$ is not empty. This holds since $\mathcal{X}_1(\spec(k)) \to \mathcal{X}_2(\spec(k))$ is essentially
 surjective.

\underline{Injective:} From Lemma \ref{inj:closed:pts} it is enough to show injectivity on $k$-points. Let $p \in V$ be a $k$-point, with its
inclusion $i:=\operatorname{Spec}(k) \to V$,
and consider its fiber $F:=\spec(k) \times_{\mathcal{X}_2} \mathcal{X}_1$.
It is enough to show that $F(\operatorname{Spec}(k))$ is equivalent to a point. Since $F$ is an algebraic space, from the definition of fibered
product of fibered categories,          
it is enough to show that
given two triples
$$(\text{Id} \in \operatorname{Spec}(k)(\operatorname{Spec}(k)),b \in \mathcal{X}_1(\operatorname{Spec}(k)), \phi:j \circ i (\Id) \to
f(b))$$
and
$$(\text{Id} \in \operatorname{Spec}(k)(\operatorname{Spec}(k)),\beta \in \mathcal{X}_1(\operatorname{Spec}(k)), \psi:j \circ i (\text{Id}) \to
f(\beta))$$
there is an isomorphism $\sigma:b \to \beta$ which makes this square commutative:
\begin{center}
$\xymatrix{\ar @{} [dr]
j\circ i (\text{Id})\ar[d]_{\text{Id}} \ar[r]^-{\phi}  & f(b) \ar[d]^{f(\sigma)} \\
j \circ i (\text{Id})\ar[r]_-{\psi} & f(\beta) }$\end{center}
Namely, it is enough to have $\sigma$ such that $f(\sigma)=\psi \circ \phi^{-1}$, and since $\mathcal{X}_1(\spec(k)) \to
\mathcal{X}_2(\spec(k))$ is fully faithful such a $\sigma$ exists.

Then $\phi$ is a bijective separated morphism from an algebraic space to a scheme. From \cite{Ols}*{Theorem 7.2.10} $V'$ is a scheme.
Then $\phi$ is an isomorphism since it is a bijective morphism of normal equidimensional schemes, and we
are over an algebraically closed field of characteristic 0.
\end{proof}

\end{document}